\documentclass{amsart}

\usepackage[left=1.5in,top=1.0in,right=1.5in]{geometry}
\usepackage{amsthm}
\usepackage{amsmath}
\usepackage{amssymb}
\usepackage{tikz}
\usetikzlibrary{decorations.text}

\usepackage{wrapfig}
\usepackage[all]{xy}
\usepackage{graphicx}
\usetikzlibrary{patterns}
\usetikzlibrary{calc}
\usepackage{mathtools}
\usepackage{enumitem}
\usepackage{bm}
\usepackage{subcaption}
\include{preamble}

\newtheorem{theorem}{Theorem}[section]
\newtheorem{corollary}[theorem]{Corollary}
\newtheorem{lemma}[theorem]{Lemma}
\newtheorem{proposition}[theorem]{Proposition}

\newtheorem{question}[theorem]{Question}

\newtheorem{notation}[theorem]{Notation}
\newtheorem{setup}[theorem]{Setup}
\newtheorem{algorithm}[theorem]{Algorithm}

\theoremstyle{remark}
\newtheorem{example}[theorem]{Example}
\newtheorem{remark}[theorem]{Remark}

\theoremstyle{definition}
\newtheorem{definition}[theorem]{Definition}

\newcommand{\RR}{\mathbb{R}}

\newcommand{\xx}{\textbf{x}}

\newcommand{\ee}{\textbf{e}}
\newcommand{\ff}{\textbf{f}}
\newcommand{\yy}{\textbf{y}}
\newcommand{\bfa}{\textbf{a}}
\newcommand{\bb}{\textbf{b}}

\newcommand{\ww}{\textbf{w}}
\newcommand{\A}{\textrm{A}}
\newcommand{\conv}{\textrm{conv}}

\newcommand{\BB}{{\mathcal B}}
\newcommand{\OO}{{\mathcal O}}
\newcommand{\Br}{\textrm{Br}}

\newcommand{\cT}{{\mathcal T}}
\newcommand{\cS}{{\mathcal S}}
\newcommand{\cD}{{\mathcal D}}

\def\fS{ {\mathfrak{S}}}

\def\Type{ {\operatorname{Type}}}
\def\Def{\operatorname{Def}}
\def\Nef{\operatorname{Nef}}
\def\Perm{\operatorname{Perm}}
\def\ncone{\operatorname{ncone}}

\def\card{\operatorname{card}}
\def\ch{\operatorname{ch}}
\def\1{ {\bf{1}}}
\def\balpha{ {\bm \alpha}}
\def\bbeta{ {\bm \beta}}

\def\ren{\!\includegraphics[scale=0.7]{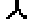}}

\numberwithin{equation}{section}

\begin{document}

\title{Deformation Cones of nested Braid fans}
\author{Federico Castillo and Fu Liu}

\address{Federico Castillo, Department of Mathematics, University of Kansas, 405 Snow Hall, 1460 Jayhawk Blvd, Lawrence, KS 66045 USA.}
\email{efecastillo.math@gmail.com}
\address{Fu Liu, Department of Mathematics, University of California, Davis, One Shields Avenue, Davis, CA 95616 USA.}
\email{fuliu@math.ucdavis.edu}
\begin{abstract}

Generalized permutohedra are deformations of regular permutohedra, and arise in many different fields of mathematics. One important characterization of generalized permutohedra is the Submodularity Theorem, which is related to the deformation cone of the Braid fan. We lay out general techniques for determining deformation cones of a fixed polytope and apply it to the Braid fan to obtain a natural combinatorial proof for the Submodularity Theorem.

We also consider a refinement of the Braid fan, called the nested Braid fan, and construct usual (respectively, generalized) nested permutohedra which have the nested Braid fan as (respectively, a coarsening of) their normal fan. We extend many results on generalized permutohedra to this new family of polytopes, including a one-to-one correspondence between faces of nested permutohedra and chains in ordered partition posets, and a theorem analogous to the Submodularity Theorem. Finally, we show that the nested Braid fan is the barycentric subdivision of the Braid fan, which gives another way to construct this new combinatorial object.

\end{abstract}

\maketitle

\section{Introduction}

Let $V$ be a finite dimensional real vector space, whose dimension we will always denote by $d.$ The \emph{dual space} $W$ of $V$ is another real vector space together with a perfect pairing $\langle\cdot,\cdot\rangle:W\times V \to \RR$.
%
A \emph{polyhedron} $P \subset V$ is the solution set of a finite set of linear inequalities:
\begin{equation}\label{ineq}
P = \{ \xx \in V \ : \ \langle \bfa_i, \xx \rangle \le b_i, \ i\in I \},
\end{equation}
where $\bfa_i$ are elements in $W$ and $b_i \in \RR$ and $I$ is a finite set of indices. By choosing bases, we can abbreviate the above system of linear inequalities as 
\begin{equation}\label{matrix}
\A \xx \le \bb,
\end{equation}
where $\A$ is the matrix whose row vectors are $\bfa_i$'s and $\bb$ is the vector with components $b_i$'s. 
A \emph{polytope} is a bounded polyhedron. A $k$-dimensional polytope $P\subset V$ is \emph{simple} if each vertex lies on exactly $k$ facets.  

In this paper, we want to study special cases of the following question: For a fixed polytope $P_0 \subset V,$ how do we characterize all ``deformations'' of $P_0$?  
In the literature, there are different equivalent definitions for what we call deformations. The initial approach we take here is to move facets of $P_0$ without passing a vertex (See Definition \ref{defn:deform0}). We also make use of an alternative definition in terms of normal fans; deformations correspond to coarsenings of the normal fan of $P_0$ (See Proposition \ref{prop:deform}). Lastly, we want to mention that this notion is equivalent (via Shephard's theorem \cite[Chapter 15, Theorem 2]{grunbaum}) to ``weak Minkowski summands'', which is central to McMullen's work on the polytope algebra (See \cite{mcmullen}).

One important family of polytopes for this paper is the family of \emph{generalized permutohedra}, which was originally introduced by Postnikov \cite[Definition 6.1]{post} as deformations of usual permutohedra. 
Generalized permutohedra contain many previously known interesting families of polytopes, including Stanley-Pitman polytopes \cite{stanley-pitman} and matroid polytopes \cite{ardila}. 
However, it turns out generalized permutohedra are translations of polymatroids (see Theorem \ref{thm:polymatroid}), which have been studied since the 70's. Polymatroids were initially defined in the context of optimization, in particular the greedy algorithm. See Edmonds' survey \cite{edmonds}, or Fujishige's book \cite{fujishige} for a more recent perspective. 
Since Postnikov's work \cite{post}, generalized permutohedra have received much research attention in the last ten years (see for example \cite{PosReiWil}, \cite{suho}, \cite{zele}). More recently, relations with Hopf monoids have been developed \cite{aguiar}.

The motivation of this article comes from two questions related to generalized permutohedra. We will discuss them in two parts below.

\subsection*{Submodularity Theorem}
One well-known result on generalized permutohedra is the Submodularity Theorem. 

\begin{definition}\label{defn:submod}
	Let $E$ be a finite set. A \emph{submodular function} is a set function $f: 2^E \to \RR$ satisfying
	\[ f(S \cup T) + f(S \cap T) \le f(S) + f(T), \quad \forall S, T \subseteq E.\]
\end{definition}

\begin{theorem}[Submodularity Theorem] \label{thm:submodular}
	There exists a bijection between generalized permutohedra of dimension at most $d$ and submodular functions $f$ on $2^{[d+1]}$ satisfying $f(\emptyset)=0$. (Here $[d+1]=\{1,2,\dots,d+1\}$.) 
\end{theorem}

Even though the Submodularity Theorem was known well before the original definition of generalized permutohedra was given by Postnikov, we couldn't find a direct reference for the statement and proof. Research papers commonly cite to \cite{post} and \cite{PosReiWil}; but it is written in neither. 
In \cite{ranktest} it appears as Proposition 15; but only the proof for one direction of the statement is provided. 
The standard classic proof we can find is in \cite[Chapter 44, Theorem 44.3]{schrijver} which has the statement in terms of polymatroids. However, the proof uses ideas from optimization, and we could not find a place that gives a clear statement of the connection between polymatroids and generalized permutohedra. 
Hence, it is still interesting to find a natural combinatorial proof for the Submodularity Theorem. 

In \cite{PosReiWil}, the authors give several equivalent definitions for generalized permutohedra, one of which states that generalized permutohedra are precisely translations of polytopes whose normal fans are coarsenings of the ``Braid fan'' $\Br_d$, which is the normal fan of the ``centralized regular permutohedron'' $\widetilde{\Pi_d}.$ (See Proposition \ref{prop:coarser}.)
As a consequence, the Submodularity Theorem is closely related to the characterization for the deformation cone of the polytope $\widetilde{\Pi_d}$ or the fan $\Br_d.$

Having this in mind, we consider the question of determining deformation cones of a general polytope $P_0$ in Section \ref{sec:prel}. 
After providing a precise definition for deformations of $P_0$ using the idea of ``moving facets without passing vertices'', we derive general techniques for computing the deformation cone of $P_0,$ using which we provide in Section \ref{sec:GP} a new combinatorial proof for Theorem \ref{thm:submodular}. 
After the notation and machinery is introduced, the proof flows naturally, which is an indication that the techniques laid out in Section \ref{sec:prel} are a good way of attacking this kind of problem. 
Another consequence of our techniques is a proof for the connection between polymatroids and generalized permutohedra.

\subsection*{The nested Braid fan} 
One characterization for generalized permutohedra is that all the edge directions are in the form of $\ee_i - \ee_j$ (See Remark \ref{rem:edges}). However, if one tries to move some facet passing a vertex, edge directions in the form of $\ee_i+\ee_j - \ee_k - \ee_\ell$ can appear. 
Therefore, we ask whether the family of generalized permutohedra can be generalized further to allow these edge directions. 
This was the original motivation for the work described in Section \ref{sec:NBF} of this article.

	\begin{figure}[t]
\begin{center}
\begin{tikzpicture}%

	\begin{scope}%
	[x={(0.249656cm, -0.577639cm)},
	y={(0.777700cm, -0.358578cm)},
	z={(-0.576936cm, -0.733318cm)},
	scale=0.300000,
	back/.style={loosely dotted, thin},
	edge/.style={color=blue!95!black, thick},
	facet/.style={fill=blue!95!black,fill opacity=0.300000},
	vertex/.style={inner sep=1pt,circle}]
%
%
\coordinate (-8.00000, -4.00000, 0.00000) at (-8.00000, -4.00000, 0.00000);
\coordinate (-8.00000, 0.00000, -4.00000) at (-8.00000, 0.00000, -4.00000);
\coordinate (-8.00000, 0.00000, 4.00000) at (-8.00000, 0.00000, 4.00000);
\coordinate (-8.00000, 4.00000, 0.00000) at (-8.00000, 4.00000, 0.00000);
\coordinate (-4.00000, -8.00000, 0.00000) at (-4.00000, -8.00000, 0.00000);
\coordinate (-4.00000, 0.00000, -8.00000) at (-4.00000, 0.00000, -8.00000);
\coordinate (-4.00000, 0.00000, 8.00000) at (-4.00000, 0.00000, 8.00000);
\coordinate (-4.00000, 8.00000, 0.00000) at (-4.00000, 8.00000, 0.00000);
\coordinate (0.00000, -8.00000, -4.00000) at (0.00000, -8.00000, -4.00000);
\coordinate (0.00000, -8.00000, 4.00000) at (0.00000, -8.00000, 4.00000);
\coordinate (0.00000, -4.00000, -8.00000) at (0.00000, -4.00000, -8.00000);
\coordinate (0.00000, -4.00000, 8.00000) at (0.00000, -4.00000, 8.00000);
\coordinate (0.00000, 4.00000, -8.00000) at (0.00000, 4.00000, -8.00000);
\coordinate (0.00000, 4.00000, 8.00000) at (0.00000, 4.00000, 8.00000);
\coordinate (0.00000, 8.00000, -4.00000) at (0.00000, 8.00000, -4.00000);
\coordinate (0.00000, 8.00000, 4.00000) at (0.00000, 8.00000, 4.00000);
\coordinate (4.00000, -8.00000, 0.00000) at (4.00000, -8.00000, 0.00000);
\coordinate (4.00000, 0.00000, -8.00000) at (4.00000, 0.00000, -8.00000);
\coordinate (4.00000, 0.00000, 8.00000) at (4.00000, 0.00000, 8.00000);
\coordinate (4.00000, 8.00000, 0.00000) at (4.00000, 8.00000, 0.00000);
\coordinate (8.00000, -4.00000, 0.00000) at (8.00000, -4.00000, 0.00000);
\coordinate (8.00000, 0.00000, -4.00000) at (8.00000, 0.00000, -4.00000);
\coordinate (8.00000, 0.00000, 4.00000) at (8.00000, 0.00000, 4.00000);
\coordinate (8.00000, 4.00000, 0.00000) at (8.00000, 4.00000, 0.00000);
\draw[edge,back] (-4.00000, -8.00000, 0.00000) -- (0.00000, -8.00000, -4.00000);
\draw[edge,back] (-4.00000, 0.00000, -8.00000) -- (0.00000, -4.00000, -8.00000);
\draw[edge,back] (0.00000, -8.00000, -4.00000) -- (0.00000, -4.00000, -8.00000);
\draw[edge,back] (0.00000, -8.00000, -4.00000) -- (4.00000, -8.00000, 0.00000);
\draw[edge,back] (0.00000, -8.00000, 4.00000) -- (4.00000, -8.00000, 0.00000);
\draw[edge,back] (0.00000, -4.00000, -8.00000) -- (4.00000, 0.00000, -8.00000);
\draw[edge,back] (0.00000, 4.00000, -8.00000) -- (4.00000, 0.00000, -8.00000);
\draw[edge,back] (4.00000, -8.00000, 0.00000) -- (8.00000, -4.00000, 0.00000);
\draw[edge,back] (4.00000, 0.00000, -8.00000) -- (8.00000, 0.00000, -4.00000);
\draw[edge,back] (8.00000, -4.00000, 0.00000) -- (8.00000, 0.00000, -4.00000);
\draw[edge,back] (8.00000, -4.00000, 0.00000) -- (8.00000, 0.00000, 4.00000);
\draw[edge,back] (8.00000, 0.00000, -4.00000) -- (8.00000, 4.00000, 0.00000);
\node[vertex] at (4.00000, 0.00000, -8.00000)     {};
\node[vertex] at (8.00000, 0.00000, -4.00000)     {};
\node[vertex] at (8.00000, -4.00000, 0.00000)     {};
\node[vertex] at (0.00000, -8.00000, -4.00000)     {};
\node[vertex] at (0.00000, -4.00000, -8.00000)     {};
\node[vertex] at (4.00000, -8.00000, 0.00000)     {};
\fill[facet] (8.00000, 4.00000, 0.00000) -- (4.00000, 8.00000, 0.00000) -- (0.00000, 8.00000, 4.00000) -- (0.00000, 4.00000, 8.00000) -- (4.00000, 0.00000, 8.00000) -- (8.00000, 0.00000, 4.00000) -- cycle {};
\fill[facet] (-8.00000, 4.00000, 0.00000) -- (-8.00000, 0.00000, -4.00000) -- (-8.00000, -4.00000, 0.00000) -- (-8.00000, 0.00000, 4.00000) -- cycle {};
\fill[facet] (4.00000, 0.00000, 8.00000) -- (0.00000, -4.00000, 8.00000) -- (-4.00000, 0.00000, 8.00000) -- (0.00000, 4.00000, 8.00000) -- cycle {};
\fill[facet] (0.00000, 8.00000, -4.00000) -- (-4.00000, 8.00000, 0.00000) -- (-8.00000, 4.00000, 0.00000) -- (-8.00000, 0.00000, -4.00000) -- (-4.00000, 0.00000, -8.00000) -- (0.00000, 4.00000, -8.00000) -- cycle {};
\fill[facet] (0.00000, 8.00000, 4.00000) -- (-4.00000, 8.00000, 0.00000) -- (-8.00000, 4.00000, 0.00000) -- (-8.00000, 0.00000, 4.00000) -- (-4.00000, 0.00000, 8.00000) -- (0.00000, 4.00000, 8.00000) -- cycle {};
\fill[facet] (4.00000, 8.00000, 0.00000) -- (0.00000, 8.00000, -4.00000) -- (-4.00000, 8.00000, 0.00000) -- (0.00000, 8.00000, 4.00000) -- cycle {};
\fill[facet] (0.00000, -4.00000, 8.00000) -- (-4.00000, 0.00000, 8.00000) -- (-8.00000, 0.00000, 4.00000) -- (-8.00000, -4.00000, 0.00000) -- (-4.00000, -8.00000, 0.00000) -- (0.00000, -8.00000, 4.00000) -- cycle {};
\draw[edge] (-8.00000, -4.00000, 0.00000) -- (-8.00000, 0.00000, -4.00000);
\draw[edge] (-8.00000, -4.00000, 0.00000) -- (-8.00000, 0.00000, 4.00000);
\draw[edge] (-8.00000, -4.00000, 0.00000) -- (-4.00000, -8.00000, 0.00000);
\draw[edge] (-8.00000, 0.00000, -4.00000) -- (-8.00000, 4.00000, 0.00000);
\draw[edge] (-8.00000, 0.00000, -4.00000) -- (-4.00000, 0.00000, -8.00000);
\draw[edge] (-8.00000, 0.00000, 4.00000) -- (-8.00000, 4.00000, 0.00000);
\draw[edge] (-8.00000, 0.00000, 4.00000) -- (-4.00000, 0.00000, 8.00000);
\draw[edge] (-8.00000, 4.00000, 0.00000) -- (-4.00000, 8.00000, 0.00000);
\draw[edge] (-4.00000, -8.00000, 0.00000) -- (0.00000, -8.00000, 4.00000);
\draw[edge] (-4.00000, 0.00000, -8.00000) -- (0.00000, 4.00000, -8.00000);
\draw[edge] (-4.00000, 0.00000, 8.00000) -- (0.00000, -4.00000, 8.00000);
\draw[edge] (-4.00000, 0.00000, 8.00000) -- (0.00000, 4.00000, 8.00000);
\draw[edge] (-4.00000, 8.00000, 0.00000) -- (0.00000, 8.00000, -4.00000);
\draw[edge] (-4.00000, 8.00000, 0.00000) -- (0.00000, 8.00000, 4.00000);
\draw[edge] (0.00000, -8.00000, 4.00000) -- (0.00000, -4.00000, 8.00000);
\draw[edge] (0.00000, -4.00000, 8.00000) -- (4.00000, 0.00000, 8.00000);
\draw[edge] (0.00000, 4.00000, -8.00000) -- (0.00000, 8.00000, -4.00000);
\draw[edge] (0.00000, 4.00000, 8.00000) -- (0.00000, 8.00000, 4.00000);
\draw[edge] (0.00000, 4.00000, 8.00000) -- (4.00000, 0.00000, 8.00000);
\draw[edge] (0.00000, 8.00000, -4.00000) -- (4.00000, 8.00000, 0.00000);
\draw[edge] (0.00000, 8.00000, 4.00000) -- (4.00000, 8.00000, 0.00000);
\draw[edge] (4.00000, 0.00000, 8.00000) -- (8.00000, 0.00000, 4.00000);
\draw[edge] (4.00000, 8.00000, 0.00000) -- (8.00000, 4.00000, 0.00000);
\draw[edge] (8.00000, 0.00000, 4.00000) -- (8.00000, 4.00000, 0.00000);
\node[vertex] at (-8.00000, -4.00000, 0.00000)     {};
\node[vertex] at (-8.00000, 0.00000, -4.00000)     {};
\node[vertex] at (-8.00000, 0.00000, 4.00000)     {};
\node[vertex] at (-8.00000, 4.00000, 0.00000)     {};
\node[vertex] at (-4.00000, -8.00000, 0.00000)     {};
\node[vertex] at (-4.00000, 0.00000, -8.00000)     {};
\node[vertex] at (-4.00000, 0.00000, 8.00000)     {};
\node[vertex] at (-4.00000, 8.00000, 0.00000)     {};
\node[vertex] at (0.00000, -8.00000, 4.00000)     {};
\node[vertex] at (0.00000, -4.00000, 8.00000)     {};
\node[vertex] at (0.00000, 4.00000, -8.00000)     {};
\node[vertex] at (0.00000, 4.00000, 8.00000)     {};
\node[vertex] at (0.00000, 8.00000, -4.00000)     {};
\node[vertex] at (0.00000, 8.00000, 4.00000)     {};
\node[vertex] at (4.00000, 0.00000, 8.00000)     {};
\node[vertex] at (4.00000, 8.00000, 0.00000)     {};
\node[vertex] at (8.00000, 0.00000, 4.00000)     {};
\node[vertex] at (8.00000, 4.00000, 0.00000)     {};
\end{scope}

	\input{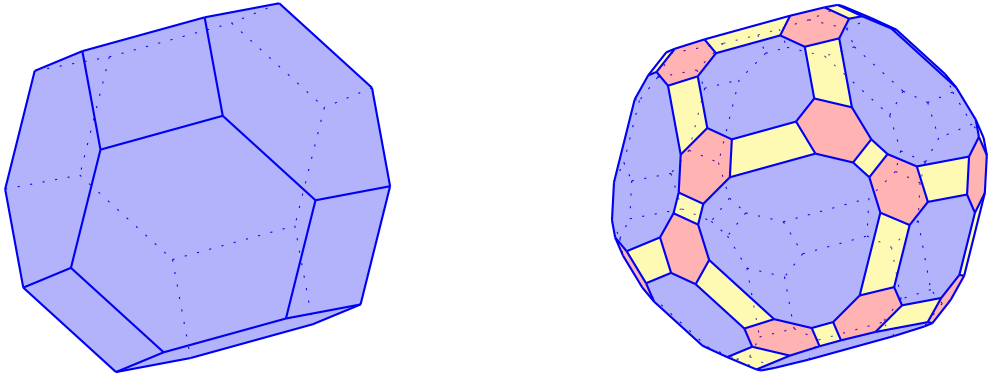}
\end{tikzpicture}
\end{center}
\caption{$\Pi_3$ and $\Pi^2_3(4,1)$}
\label{fig:2polytopes}
\end{figure}

The maximal cones in the Braid fan $\Br_d$ are sets of points whose coordinates are given in a fixed order. In Section \ref{sec:NBF}, we introduce the ``nested Braid fan'' $\Br_d^2,$ which is a refinement of $\Br_d$ by considering first differences of ordered coordinates. (See Definitions \ref{defn:NBF} and \ref{defn:NBF2} for detail.) We show that $\Br_d^2$ is the normal fan of ``usual nested permutohedra'', a subfamily of which is $\Pi_d^2(M,N)$ (called ``regular nested permutohedra''), and thus is a projective fan. (See Figure \ref{fig:2polytopes} for a picture of $\Pi_3$ and $\Pi_3^2(4,1)$ side by side.) We then use the general techniques derived in Section \ref{sec:prel} to give a characterization for the deformation cone of $\Br_d^2$ analogous to the results for the deformation cone of $\Br_d.$

One key ingredient in our proof for the Submodularity Theorem is the natural one-to-one correspondence between chains in the Boolean algebra $\BB_{d+1}$ and faces of a regular permutohedron.
Parallelly, in Section \ref{sec:NBF}, we consider the ``ordered partition poset'' $\OO_{d+1}$ (see Definition \ref{defn:osp}), and show same statement holds for $\OO_{d+1}$ and the regular nested permutohedron. 

We remark that there are multiple possible answers to our original question of how to generalize generalized permutohedra further. We landed on our current construction of nested permutohedra after trying a few possible approaches, as it has the richest combinatorics among all constructions we considered. In fact, this family of polytopes turned out to be more interesting than we had expected, having connections to combinatorial objects previously studied by other authors.  
For example, our construction has the same flavor as the construction of the ``permuto-associahedron'', defined as a CW complex by Kapranov in \cite{kapranov} and as a polytope by Reiner and Ziegler in \cite{ReinerZieg} using the theory of fiber polytopes. 
The rough idea for the construction of a permuto-associahedron given in \cite{ReinerZieg} is to put an associahedron on each vertex of a permutohedron, which is parallel to one way of constructing a nested permutohedron, by putting a permutohedron on each vertex of a (one dimensional higher) permutohedron.
One difference between Reiner-Ziegler's work and ours is that we are able to give explicit coordinates for our construction. Since the associahedron can be realized as a deformation of the permutohedron \cite[Section 8.2]{post}, we expect that the permuto-associahedron will have a concrete realization as a deformation of this nested permutohedron. 


Finally, during a talk given by the first author on materials presented in Sections \ref{sec:GP} and \ref{sec:NBF}, Victor Reiner asked whether the nested Braid fan is the barycentric subdivision of the Braid fan. We give an affirmative answer to his question in Section \ref{sec:chisel}. 

\subsection*{Organization of the paper} In \S \ref{sec:prel}, we will present/review definitions of deformation cones of polytopes and projective fans, and discuss general techniques for computing them from the polytopal side. 
In \S \ref{sec:GP}, we review known facts about generalized permutohedra, apply techniques derived in \S \ref{sec:prel} to find deformation cones of $\Br_d$, and give a proof for the Submodularity Theorem. In \S \ref{sec:NBF}, we define nested Braid fan $\Br_d^2$ and nested permutohedra, and discuss their combinatorics, using which we give inequality description for nested permutohedra and determine the deformation cone of $\Br_d^2.$ 
In \S \ref{sec:chisel}, we describe how we can obtain the nested Braid fan as the barycentric subdivision of the Braid fan, answering Victor Reiner's question.
We finish the main body of this article with some questions that might be interesting for future research in \S \ref{sec:question}. 


\subsection*{Acknowledgements} The second author is partially supported by National Science Foundation grant
DMS-1265702 and a grant from the Simons Foundation \#426756. The final writing of the work was completed when both authors were attending the program ``Geometric and Topological Combinatorics'' at the Mathematical Sciences Research Institute in Berkeley, California, during the Fall 2017 semester, and they were partially supported by the National Science Foundation grant DMS-1440140.

The authors would like to thank Federico Ardila and Brian Osserman for helpful discussion, and thank Alex Fink and Christian Haase for explaining Theorem \ref{thm:bary1}. 

\section{Determining deformation cones} \label{sec:prel}
We assume familiarity with basic definitions of polyhedra and polytopes as presented in \cite{barvinok, zie}. 
The main purpose of this section is to derive a systematic way to answer the following general question: For a fixed polytope $P_0 \subset V,$ how do we characterize all ``deformations'' of $P_0$?   
We start by setting up our question formally.
\begin{setup}\label{setup1}
	Let $P_0$ be a fixed full-dimensional polytope in $V$ defined by $\A \xx \le \bb_0$, where each inequality is \emph{facet-defining}, i.e., $\{ \xx \in P \ : \ \langle \bfa_i, \xx \rangle = b_i\}$
	is a facet of $P.$ 
Suppose $P_0$ has $n$ facets $F_1, \dots, F_m$. We may assume that
	the system defining $P$ is 
\begin{equation}\label{fineq}
	\langle \bfa_i, \xx \rangle \le b_{0,i}, \quad 1 \le i \le m,
\end{equation}
where $\bfa_i$ is an normal vector to the facet $F_i.$
\end{setup}

Roughly speaking, a deformation of $P_0$ is a polytope obtained from $P_0$ by moving facets of $P_0$ ``without passing any vertices''. We make this more precise below. 
\begin{definition} \label{defn:deform0}
	A polytope $Q \subset V$ is a \emph{deformation} of $P_0$ (described in Setup \ref{setup1}), 
	if there exists $\bb \in \RR^m$ such that the following two conditions are satisfied:
	\begin{enumerate}[label=(\alph*)]
		\item \label{item:pts}
	$Q$ is defined by $\A \xx \le \bb$ (with the same matrix $\A$ as in Setup \ref{setup1}) or equivalently, 
\begin{equation}\label{Qfineq}
	\langle \bfa_i, \xx \rangle \le b_{i}, \quad 1 \le i \le m. 
\end{equation}
\item \label{item:nopass} For any vertex $v$ of $P_0$, if $F_{i_1}, F_{i_2}, \dots, F_{i_k}$ are the facets of $P_0$ on which $v$ lies, then the intersection of
	\[
		\{ \xx \in V \ : \ \langle \bfa_{i_j}, \xx \rangle = b_{i_j}\}, \quad 1 \le j \le k
\]
is a vertex $u$ of $Q.$
	\end{enumerate}
	We call $\bb$ a \emph{deforming vector} for $Q$.
\end{definition}

It is not hard to see that any deformation $Q$ of $P_0$ is associated with a \emph{unique} deforming vector $\bb$ because conditions \ref{item:pts} and \ref{item:nopass} imply that the entries of $\bb$ must satisfy
\[ b_i = \max_{\xx \in Q} \langle \bfa_i, \xx \rangle, \quad \forall 1 \le i \le m. \]
Thus, we say $\bb$ is \emph{the} deforming vector for $Q.$
The uniqueness of $\bb$, together with condition (a), establishes a one-to-one correspondence between deformations $Q$ of $P_0$ and their associated deformation vectors. Therefore, we give the following definition.
\begin{definition}
	The \emph{deformation cone} of $P_0$, denoted by $\Def(P_0)$, is the collection of deforming vectors $\bb \in \RR^m$ described in Definition \ref{defn:deform0}. 
\end{definition}

\begin{figure}[t]
\begin{tikzpicture}
	\begin{scope}[scale=0.5]

\node at (13,0) {$\left(\begin{array}{rr} -1&0 \\ 0&1 \\ 0&-1 \\ 1&-1 \end{array}\right)\left(\begin{array}{r} x\\y\end{array}\right) \leq \left(\begin{array}{r} 1 \\ 2 \\ 1 \\ 2 \end{array}\right).$};

\draw[help lines,dashed] (3,0)--(-3,0);
\draw[help lines,dashed] (0,3)--(0,-3);

\draw (-1,-3)--(-1,3);
\draw (-3,2)--(5,2);
\draw (-3,-1)--(3,-1);
\draw (0,-2)--(5,3);

\node[above left] at (-1,2) {$\scriptstyle (-1,2)$};
\node[below left] at (-1,-1) {$\scriptstyle (-1,-1)$};
\node[below right] at (1,-1) {$\scriptstyle (1,-1)$};
\node[above right] at (4,2) {$\scriptstyle (4,2)$};

\draw[fill=red] (-1,2) circle [radius=0.1];
\draw[fill=red] (-1,-1) circle [radius=0.1];
\draw[fill=red] (1,-1) circle [radius=0.1];
\draw[fill=red] (4,2) circle [radius=0.1];

\node at (0.8, 0.7) {\Large $\textcolor{blue}{P_0}$};
\node[above left] at (-1,0) {$F_1$};
\node[above] at (2,2) {$F_2$};
\node[below left] at (0,-1) {$F_3$};
\node[above right] at (3,0) {$F_4$};

\draw[pattern=dots] (-1,-1)--(1,-1)--(4,2)--(-1,2)--cycle;
\end{scope}

 \begin{scope}[yshift=-2.8cm, scale=0.35]
\draw[help lines,dashed] (3,0)--(-3,0);
\draw[help lines,dashed] (0,3)--(0,-3);

\draw (-3,-3)--(-3,3);
\draw (-3.5,2)--(5,2);
\draw (-3.5,0)--(3,0);
\draw (0.5,-2)--(5.5,3);

\node[above left] at (-3,2) {$\scriptstyle (-1,2)$};
\node[below left] at (-3,0) {$\scriptstyle (-3,0)$};
\node[below right] at (2.5,0) {$\scriptstyle (2.5,0)$};
\node[above right] at (4.5,2) {$\scriptstyle (4.5,2)$};

\draw[fill=red] (-3,2) circle [radius=0.2];
\draw[fill=red] (-3,0) circle [radius=0.2];
\draw[fill=red] (2.5,0) circle [radius=0.2];
\draw[fill=red] (4.5,2) circle [radius=0.2];

\draw[pattern=dots] (-3,0)--(2.5,0)--(4.5,2)--(-3,2)--cycle;

\node at (0,-4) {$\bb_1=(3,2,0,2.5)$};
\node at (0.5,2.7) {\Large $\textcolor{blue}{Q_1}$};

\end{scope}

 \begin{scope}[xshift= 5cm, yshift=-2.8cm, scale=0.35]
\draw[help lines,dashed] (3,0)--(-3,0);
\draw[help lines,dashed] (0,3)--(0,-3);

\draw (-1,-3)--(-1,3);
\draw (-3,2)--(5,2);
\draw (-3,-1)--(3,-1);
\draw (-2,-2)--(3,3);

\node[above left] at (-1,2) {$\scriptstyle (-1,2)$};
\node[below left] at (-1.2,-0.3) {$\scriptstyle (-1,-1)$};
\node[above right] at (2,2) {$\scriptstyle (2,2)$};

\draw[fill=red] (-1,2) circle [radius=0.2];
\draw[fill=red] (-1,-1) circle [radius=0.2];
\draw[fill=red] (2,2) circle [radius=0.2];
\draw[pattern=dots] (-1,-1)--(2,2)--(-1,2)--cycle;

\node at (0,-4) {$\bb_2=(1,2,1,0)$};
\node at (0.5,2.7) {\Large $\textcolor{blue}{Q_2}$};
\end{scope}

\begin{scope}[xshift= 9.5cm, yshift=-2.8cm, scale=0.35]
\draw[help lines,dashed] (3,0)--(-3,0);
\draw[help lines,dashed] (0,3)--(0,-3);

\draw (-1,-3)--(-1,3);
\draw (-3,2)--(5,2);
\draw (-3,-2)--(3,-2);
\draw (-2,-2)--(3,3);

\node[above left] at (-1,2) {$\scriptstyle (-1,2)$};
\node[below left] at (-1.2,-0.3) {$\scriptstyle (-1,-1)$};
\node[above right] at (2,2) {$\scriptstyle (2,2)$};
\node[below right] at (-1,-2) {$\scriptstyle (-1,-2)$};
\node[below left] at (-1.5,-2) {$\scriptstyle (-2,-2)$};

\draw[fill=red] (-1,2) circle [radius=0.2];
\draw[fill=black] (-1,-1) circle [radius=0.2];
\draw[fill=red] (2,2) circle [radius=0.2];
\draw[fill=red] (-1,-2) circle [radius=0.2];
\draw[fill=red] (-2,-2) circle [radius=0.2];
\draw[pattern=dots] (-1,-1)--(2,2)--(-1,2)--cycle;

\node at (0,-4) {$\bb_3=(1,2,2,0)$};
\node at (0.5,2.7) {\Large $\textcolor{blue}{Q_3}$};
\end{scope}
\end{tikzpicture}
\caption{Polytopes in Example \ref{ex:example}}
\label{fig:exfigure}
\end{figure}
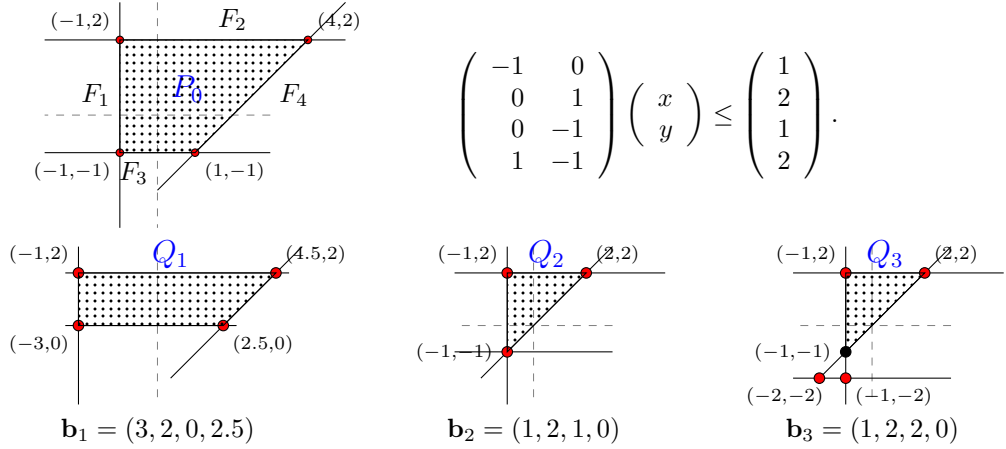

\begin{example}\label{ex:example} 
Let $P_0 \subset V = \RR^2$ be the polytope on the top left of Figure \ref{fig:exfigure}, which is defined by the linear system given to its right. Let $\A$ be the matrix in the linear system. 
Any deformation $Q$ of $P_0$ can be defined by $\A \xx \le \bb$ for some $\bb.$ 
Two possible deformations $Q_1$ and $Q_2$ together with their respective deforming vectors $\bb_1$ and $\bb_2$ are shown on the bottom of Figure \ref{fig:exfigure}. Notice that $Q_3$ defined by $\A \xx \le \bb_3$ is exactly the same polytope as $Q_2$, so is a deformation of $P_0$. However, $\bb_3$ does not satisfy condition \ref{item:nopass}, and thus is not a deforming vector. Hence, $\bb_1, \bb_2 \in \Def(P_0)$, but $\bb_3 \not\in \Def(P_0).$ (This conclusion will be proved formally in Example \ref{ex:exampleprop}.)
\end{example}

The deformation cone of $P_0$ is a natural subject to study if one is interested in deformations of the fixed polytope $P_0$. We can now rephrase our initial general question.
\begin{question}\label{ques1}
Fix a full dimensional polytope $P_0 \subset V$. How do we find a characterization for $\Def(P_0)?$
\end{question}

There is another equivalent way of defining deformations $Q$ of $P_0$ using normal fans of polytopes. (See Definition \ref{defn:normal} for a formal definition of normal cones and normal fans.) 
\begin{proposition}\label{prop:deform}
	A polytope $Q \subset V$ is a deformation of $P_0$ if and only if the normal fan $\Sigma(Q)$ of $Q$ is a coarsening of the normal fan $\Sigma(P_0)$ of $P_0$.

\end{proposition}

Proposition \ref{prop:deform} is well known; but for ease of reference, we provide a proof for it. Since this proof is quite different from what we discuss in the rest of the paper, it will be included in Appendix \ref{apd:normal}.


Note that in Example \ref{ex:example}, the polytope $Q_1$ has the same normal fan as $P_0$, whereas $Q_2$'s normal fan is a coarsening.

Proposition \ref{prop:deform} implies that if two polytopes $P_1$ and $P_2$ have the same normal fan $\Sigma$, they have exactly the same deformation cone. By abusing the notation, we might denote this deformation cone by  $\Def(\Sigma),$ and call it the \emph{deformation cone} of $\Sigma.$

We say a fan $\Sigma$ is \emph{projective} if it is the normal fan of a polytope. (It is not true all the fans are projective.) Once we know that a projective fan $\Sigma$ is the normal fan of a polytope, one can check that the polytope is full dimensional if and only if $0 \in \Sigma,$ i.e., all cones in $\Sigma$ are pointed. We use these language to rewrite Setup \ref{setup1} and Question \ref{ques1}

\begin{setup}\label{setup2}
Let $\Sigma_0$ be a projective fan in $W$ such that $0 \in \Sigma_0$. Assume it has $m$ one dimensional cones that are generated by rays $\bfa_1, \dots, \bfa_m,$ respectively. 
\end{setup}
\begin{question}\label{ques2}
	Given a fixed fan $\Sigma_0$ as described in Setup \ref{setup2},
how do we find a characterization for $\Def(\Sigma_0)?$
\end{question}

Questions \ref{ques1} and \ref{ques2} are the same question in two different languages, and both have been studied, 
where the latter one is related to the study of toric varieties. (See \cite{cls} for general results on toric varieties.)
It is worth remarking that big part of the motivation and tools come from that branch of mathematics. When $\Sigma_0$ is smooth, then $\Def(\Sigma_0)$, modulo its linearity space, is isomorphic to $\textrm{Nef}(\Sigma_0)$, the cone of \emph{numerically effective divisors} (see \cite[Chapter 6]{cls}).

\begin{remark}
	In addition to the two definitions we have provided, there are additional different but equivalent ways of defining deformations of polytopes. In particular, in the Appendix of \cite{PosReiWil}, the authors discuss five different ways, including the normal fan version stated in Proposition \ref{prop:deform}. However, they restrict their definitions to simple polytopes only, while our definition is for \emph{any} polytope. Furthermore, it seems our Definition \ref{defn:deform0} has not (or at least not explicitly) appeared in the literature, and actually is very important for determining deformation cones as the techniques (that will be shown below) are derived from it directly. 
\end{remark}


There are three main results that will be presented in the rest of this section. The first result is Corollary \ref{cor:ineqs}, in which we give an explicit description for the deformation cone $\Def(P_0)$ of $P_0$ using linear equalities and inequalities. This will be derived directly from Definition \ref{defn:deform0}. We then analyze inequalities in Corollary \ref{cor:ineqs} further and apply it to simple polytopes to obtain in Proposition \ref{prop:redux} a simpler description for $\Def(P_0)$ using inequalities indexed by edges of $P_0$. We then give our third result - Proposition \ref{prop:reduxfan} - by restating Proposition \ref{prop:redux} using the language of simplicial fans, in which inequalities are indexed by pairs of adjacent maximal cones in the fan. We end this section with a discussion on how to determine whether a polytope is a deformation of $P_0$ using $\Def(P_0).$

\subsection*{Deformation cones of (not necessarily simple) polytopes}

Even though condition \ref{item:pts} of Definition \ref{defn:deform0} is necessary for the definition of deformations of a fixed polytope, 
if one is only concerned about deforming vectors, only condition \ref{item:nopass} is needed as stated in Lemma \ref{lem:detb} below.

\begin{definition}
Suppose $v$ is a vertex of $P_0$. Let $F_{i_1}, F_{i_2}, \dots, F_{i_k}$ be the facets of $P_0$ on which $v$ lies and let $u_\bb$ be the intersection of
	\[
		\{ \xx \in V \ : \ \langle \bfa_{i_j}, \xx \rangle = b_{i_j}\}, \quad 1 \le j \le k.
\]

We say a vector $\bb \in \RR^m$ satisfies the \emph{NEI} (short for ``non-empty-intersection'') \emph{condition for $v$} if $u_\bb$ is nonempty, so is a point.  
	
We say $\bb$ satisfies the \emph{no-passing condition for $v$} if
$\A u_\bb \le \bb,$ or equivalently, 
	\[ \langle \bfa_i, u_\bb \rangle \le b_i, \quad \forall i \neq i_1,\dots, i_k.\]
\end{definition}

\begin{lemma}\label{lem:detb}
	Let $\bb \in \RR^m.$ Then $\bb \in \Def(P_0)$ if and only if $\bb$ satisfies the NEI and no-passing conditions for every vertex $v$ of $P_0.$
%
\end{lemma}

\begin{proof}
	The forward implication follows directly from Definition \ref{defn:deform0}. Conversely, suppose the two conditions hold. Let $Q$ be defined by $\A \xx \le \bb.$ Condition \ref{item:pts} of Definition \ref{defn:deform0} is automatically satisfied, and the NEI and no-passing conditions guarantee that $u_\bb$ is a vertex of $Q$, and thus condition \ref{item:nopass} holds. 
\end{proof}

The no-passing condition can fail in different scenarios. For the polytope $Q_3$ of Example \ref{ex:example}, not only the inequality $-y \le 2$ is not facet-defining, but the hyperplane determined by $-y = 2$ does not ``touch'' $Q_3,$ which causes the failure of the no-passing condition.
Below, we show a different example where the no-passing condition fails even though all the inequalities are still facet-defining.

\begin{example}
	Consider the $3$-dimensional polytopes $P_0$ and $Q$ shown on the left of Figure \ref{fig:above}. The right of Figure \ref{fig:above} shows how they look like when being viewed from above.
\begin{figure}[!htb]
\begin{tikzpicture}%
	\begin{scope}
	[x={(-0.2 cm, -0.1 cm)},
	y={(0.5 cm,0 cm)},
	z={(0cm, 0.5cm)},
	scale=0.5,
	back/.style={loosely dotted, thin},
	edge/.style={color=black!95!black, thick},
	facet/.style={fill=black!95!black,fill opacity=0.200},
	vertex/.style={inner sep=1pt,circle,draw=red!25!black,fill=red!75!black,thick,anchor=base}]
%
%
\coordinate (8.00000, 0.00000, 0.00000) at (8.00000, 0.00000, 0.00000);
\coordinate (8.00000, 10.00000, 0.00000) at (8.00000, 10.00000, 0.00000);
\coordinate (0.00000, 10.00000, 0.00000) at (0.00000, 10.00000, 0.00000);
\coordinate (4.00000, 6.00000, 4.00000) at (4.00000, 6.00000, 4.00000);
\coordinate (4.00000, 4.00000, 4.00000) at (4.00000, 4.00000, 4.00000);
\coordinate (0.00000, 0.00000, 0.00000) at (0.00000, 0.00000, 0.00000);
\fill[facet] (4.00000, 4.00000, 4.00000) -- (8.00000, 0.00000, 0.00000) -- (8.00000, 10.00000, 0.00000) -- (4.00000, 6.00000, 4.00000) -- cycle {};
\fill[facet] (4.00000, 6.00000, 4.00000) -- (8.00000, 10.00000, 0.00000) -- (0.00000, 10.00000, 0.00000) -- cycle {};
\fill[facet] (0.00000, 0.00000, 0.00000) -- (0.00000, 10.00000, 0.00000) -- (4.00000, 6.00000, 4.00000) -- (4.00000, 4.00000, 4.00000) -- cycle {};
\fill[facet] (0.00000, 0.00000, 0.00000) -- (8.00000, 0.00000, 0.00000) -- (4.00000, 4.00000, 4.00000) -- cycle {};
\draw[edge] (8.00000, 0.00000, 0.00000) -- (8.00000, 10.00000, 0.00000);
\draw[edge] (8.00000, 0.00000, 0.00000) -- (4.00000, 4.00000, 4.00000);
\draw[edge] (8.00000, 0.00000, 0.00000) -- (0.00000, 0.00000, 0.00000);
\draw[edge] (8.00000, 10.00000, 0.00000) -- (0.00000, 10.00000, 0.00000);
\draw[edge] (8.00000, 10.00000, 0.00000) -- (4.00000, 6.00000, 4.00000);
\draw[edge] (0.00000, 10.00000, 0.00000) -- (4.00000, 6.00000, 4.00000);
\draw[edge] (0.00000, 10.00000, 0.00000) -- (0.00000, 0.00000, 0.00000);
\draw[edge] (4.00000, 6.00000, 4.00000) -- (4.00000, 4.00000, 4.00000);
\draw[edge] (4.00000, 4.00000, 4.00000) -- (0.00000, 0.00000, 0.00000);
\node[vertex] at (8.00000, 0.00000, 0.00000)     {};
\node[vertex] at (8.00000, 10.00000, 0.00000)     {};
\node[vertex] at (0.00000, 10.00000, 0.00000)     {};
\node[vertex] at (4.00000, 6.00000, 4.00000)     {};
\node[vertex] at (4.00000, 4.00000, 4.00000)     {};
\node[vertex] at (0.00000, 0.00000, 0.00000)     {};
\node at (-6,-4) {$\textcolor{blue}{P_0}$};
\end{scope}
  \begin{scope}%
	[xshift=3.5cm, x={(-0.2 cm, -0.1 cm)},
	y={(0.6 cm,0 cm)},
	z={(0cm, 0.4cm)},
	scale=0.5,
	back/.style={loosely dotted, thin},
	edge/.style={color=black!95!black, thick},
	facet/.style={fill=black!95!black,fill opacity=0.200},
	vertex/.style={inner sep=1pt,circle,draw=red!25!black,fill=red!75!black,thick,anchor=base}]
%
%
\coordinate (10.00000, 0.00000, 0.00000) at (10.00000, 0.00000, 0.00000);
\coordinate (10.00000, 8.00000, 0.00000) at (10.00000, 8.00000, 0.00000);
\coordinate (6.00000, 4.00000, 4.00000) at (6.00000, 4.00000, 4.00000);
\coordinate (4.00000, 4.00000, 4.00000) at (4.00000, 4.00000, 4.00000);
\coordinate (0.00000, 0.00000, 0.00000) at (0.00000, 0.00000, 0.00000);
\coordinate (0.00000, 8.00000, 0.00000) at (0.00000, 8.00000, 0.00000);
\fill[facet] (6.00000, 4.00000, 4.00000) -- (10.00000, 0.00000, 0.00000) -- (10.00000, 8.00000, 0.00000) -- cycle {};
\fill[facet] (0.00000, 8.00000, 0.00000) -- (10.00000, 8.00000, 0.00000) -- (6.00000, 4.00000, 4.00000) -- (4.00000, 4.00000, 4.00000) -- cycle {};
\fill[facet] (0.00000, 8.00000, 0.00000) -- (4.00000, 4.00000, 4.00000) -- (0.00000, 0.00000, 0.00000) -- cycle {};
\fill[facet] (0.00000, 0.00000, 0.00000) -- (10.00000, 0.00000, 0.00000) -- (6.00000, 4.00000, 4.00000) -- (4.00000, 4.00000, 4.00000) -- cycle {};
\draw[edge] (10.00000, 0.00000, 0.00000) -- (10.00000, 8.00000, 0.00000);
\draw[edge] (10.00000, 0.00000, 0.00000) -- (6.00000, 4.00000, 4.00000);
\draw[edge] (10.00000, 0.00000, 0.00000) -- (0.00000, 0.00000, 0.00000);
\draw[edge] (10.00000, 8.00000, 0.00000) -- (6.00000, 4.00000, 4.00000);
\draw[edge] (10.00000, 8.00000, 0.00000) -- (0.00000, 8.00000, 0.00000);
\draw[edge] (6.00000, 4.00000, 4.00000) -- (4.00000, 4.00000, 4.00000);
\draw[edge] (4.00000, 4.00000, 4.00000) -- (0.00000, 0.00000, 0.00000);
\draw[edge] (4.00000, 4.00000, 4.00000) -- (0.00000, 8.00000, 0.00000);
\draw[edge] (0.00000, 0.00000, 0.00000) -- (0.00000, 8.00000, 0.00000);
\node[vertex] at (10.00000, 0.00000, 0.00000)     {};
\node[vertex] at (10.00000, 8.00000, 0.00000)     {};
\node[vertex] at (6.00000, 4.00000, 4.00000)     {};
\node[vertex] at (4.00000, 4.00000, 4.00000)     {};
\node[vertex] at (0.00000, 0.00000, 0.00000)     {};
\node[vertex] at (0.00000, 8.00000, 0.00000)     {};
\node at (-6,-4) {$\textcolor{blue}{Q}$};
\end{scope}

\begin{scope}[xshift=7cm, yshift=-1cm, scale=0.2]
\draw (0,0)--(10,0)--(10,8)--(0,8)--cycle;
\draw (0,0)--(4,4)--(6,4)--(10,0);
\draw (0,8)--(4,4)--(6,4)--(10,8);
\draw[fill] (0,0) circle [radius=0.2];
\draw[fill] (10,0) circle [radius=0.2];
\draw[fill] (10,8) circle [radius=0.2];
\draw[fill] (0,8) circle [radius=0.2];
\draw[fill] (4,4) circle [radius=0.2];
\draw[fill] (6,4) circle [radius=0.2];
\node at (1.5,4) {$\textcolor{blue}{P_0}$};
\end{scope}

\begin{scope}[xshift=10cm, yshift=-1cm, scale=0.2]
\draw (0,0)--(0,10)--(8,10)--(8,0)--cycle;
\draw (0,0)--(4,4)--(4,6)--(0,10);
\draw (8,0)--(4,4)--(4,6)--(8,10);
\draw[fill] (0,0) circle [radius=0.2];
\draw[fill] (0,10) circle [radius=0.2];
\draw[fill] (8,10) circle [radius=0.2];
\draw[fill] (8,0) circle [radius=0.2];
\draw[fill] (4,4) circle [radius=0.2];
\draw[fill] (4,6) circle [radius=0.2];

\node at (2,5) {$\textcolor{blue}{Q}$};
\end{scope}

\end{tikzpicture}
\caption{}
\label{fig:above}
\end{figure}
$Q$ is obtained from $P_0$ by moving the $\textsc{left}$ and $\textsc{right}$ facets of $P_0$ inward ``too much''. 
Notice that in $P_0$, the facets $\textsc{front},\textsc{back}$ and $\textsc{right}$ intersect in a vertex, but in $Q$ they do not. More precisely, the hyperplanes determined by the $\textsc{front},\textsc{back}$ and $\textsc{right}$ facets of $Q$ intersect at a point outside of $Q$, and thus is on the wrong side of the hyperplane determined by the $\textsc{left}$ facet. So $Q$ is not a deformation of $P_0$ even though it can be defined using the same matrix $\A$ as $P_0.$ 
\end{example}

It is straightforward to translate conditions in Lemma \ref{lem:detb} to explicit linear conditions. We give the following notation and definition before stating Corollary \ref{cor:ineqs}.
\begin{notation}
For convenience, for any facet $F = F_i$ of $P_0,$ we sometimes use $F$ as the subscripts for $\bfa_i$ and $b_i,$ that is
\[ \bfa_F = \bfa_i, \quad b_F = b_i.\]
\end{notation}

\begin{definition}
Let $v$ be a vertex of $P_0.$
Suppose $F_{i_1}, F_{i_2}, \dots, F_{i_k}$ (with $i_1 < i_2 < \cdots < i_k$) are the facets of $P_0$ on which $v$ lies. (Note that we must have $k \ge d$.) We say $F_{i_1},\dots, F_{i_d}$ are the \emph{first $d$ supporting facets of $v$}, and $F_{i_j}$ for $d < j \le k$ is an \emph{extra supporting facet of $v$}. (Note these definitions rely on the specific ordering we give for facets of $P_0$.)

For any $\bb \in \RR^m$, let $v_\bb$ be the intersection of the hyperplanes determined by the first $d$ supporting facets of $v,$ that is, $v_\bb$ is the intersection of
	\[
		\{ \xx \in V \ : \ \langle \bfa_{i_j}, \xx \rangle = b_{i_j}\}, \quad 1 \le j \le d,
\]
	which clearly is a point.

For any vertex $v$ of $P_0$ and any facet $F$ of $P_0$, we associate with the pair $(v, F)$ an equality or an inequality as below: 
\begin{align*}
E_{v,F}(\bb) :& \quad \langle \bfa_F, v_\bb \rangle = b_F, \\
I_{v,F}(\bb) :& \quad \langle \bfa_F, v_\bb \rangle \le b_F. 
\end{align*}
\end{definition}

\begin{corollary}\label{cor:ineqs}
The deformation cone $\Def(P_0)$ is the collection of vectors $\bb$ satisfying the following two conditions: 
\begin{enumerate}[label=(\roman*)]
	\item All the equalities $E_{v,F}(\bb)$ hold, where $(v,F)$ is a vertex-facet pair of $P_0$ such that $F$ is an extra supporting facet of $v$.
	\item 
All the inequalities $I_{v,F}(\bb)$ holds, where $(v,F)$ is a vertex-facet pair of $P_0$ such that $F$ is not a supporting facet of $v.$
\end{enumerate}

Therefore, $\Def(P_0)$ is (indeed) a polyhedral cone.
\end{corollary}

\begin{proof}
One sees that condition (i) is equivalent to the NEI condition, and condition (ii) is equivalent to the no-passing condition. Moreover, since $v_\bb$ is the solution of a linear system, it is written as a linear combinations of entries in $\bb.$ Therefore, each equality or inequality is linear. So the solution set of $\bb$ is a polyhedral cone.
\end{proof}
\begin{remark}\label{rem:nef}
The deformation cone is related to the $\textrm{Nef}$ cone (see \cite[Definition 6.3.18]{cls}) of the toric variety associated with $\Sigma(P)$, as follows. Any polytope whose normal fan is a coarsening of $\Sigma(P)$ gives a \emph{basepoint free} divisor, which for toric varieties is the same as \emph{nef} (\cite[Theorem 6.3.12]{cls}) divisor. The difference is that the $\textrm{Nef}$ cone does not distinguish between translations of the same polytope, since they give the same divisor modulo rational equivalence. Hence, the $\textrm{Nef}$ cone is isomorphic to the deformation cone modulo translations.
\end{remark}

The number of inequalities in Corollary \ref{cor:ineqs} can be reduced. Given a polytope $P_0$, we say a facet $F$ is a \emph{neighbor} of a vertex $v$ and $(v,F)$ is a \emph{neighboring pair} of $P_0$, if $v\notin F$ but there exist a vertex $v'\in F$ such that $\{v, v'\}$ is an edge of $P_0$.

\begin{proposition}\label{prop:preredux}
Let $\bb \in \RR^m.$ The following are equivalent.
\begin{enumerate}
	\item $\bb \in \Def(P_0).$
	\item 
		Conditions (i) and (ii) of Corollary \ref{cor:ineqs} are satisfied.
	\item Condition (i) of Corollary \ref{cor:ineqs} is satisified, and all the inequalities $I_{v,F}(\bb)$ are satisfied, where $(v, F)$ is a neighboring vertex-facet pair of $P_0$. 
	\item Condition (i) of Corollary \ref{cor:ineqs} is satisified, and for any edge $e=\{v, v'\}$ of $P_0$, there exists $\lambda_e \in \RR_{\ge 0}$ such that $v- v' = \lambda_e (v_\bb - v'_\bb).$	
\end{enumerate}
\end{proposition}

\begin{proof}
	The equivalence between (1) and (2) is assured by Corollary \ref{cor:ineqs}, and it is clear that (2) implies (3). So it suffices to show (3) implies (4) and (4) implies (2).

\noindent	\underline{``$(3) \implies (4)$'':} 
%
There exist $d-1$ facets $F_{j_1},\dots, F_{j_{d-1}}$ of $P_0,$ such that the edge $e=\{v, v'\}$ in $P_0$ is the intersection of them. Thus $v-v'$ is in the one dimensional space that is orthogonal to the $(d-1)$-space spanned by $\bfa_{j_1},\dots, \bfa_{j_{d-1}}.$ By the definition of $v_\bb$ and $v'_\bb$ and because condition (i) of Corollary \ref{cor:ineqs} is satisfied, one sees that $v_\bb-v'_\bb$ should be in the same one dimensional space. Therefore, $v-v' = \lambda_e(v_\bb - v'_\bb)$ for some $\lambda_e \in \RR.$ Thus, it is left to show that $\lambda_e \ge 0.$

	Let $F$ be a facet that $v'$ lies on but $v$ does not. Since $v \in P_0$  has to satisfy the strict inequality in $I_{v,F}(\bb_0),$ we have that
	$\langle \bfa_F, v \rangle < b_{0, F} = \langle \bfa_F, v' \rangle,$
	which is equivalent to 
		$\langle \bfa_F, v-v' \rangle < 0.$
	On the other hand, as $(v,F)$ is a neighboring pair, we also have $I_{v,F}(\bb)$ holds, which is equivalent to 
	$\langle \bfa_F, v_\bb -v'_\bb \rangle = \lambda_e \langle \bfa_F, v-v' \rangle \le 0.$	Hence, $\lambda_e \ge 0.$

	\noindent	\underline{``$(4) \implies (2)$'':} 
	Let $(v, F)$ be a vertex-facet pair of $P_0$ such that $v$ does not lie on $F.$ 
	Let $v_0 = v$ and pick a point $\xx \in F.$ Then 
	\[ \langle \bfa_F, v_0 \rangle < b_{0,F} = \langle \bfa_F, \xx \rangle,\]
	Since $\xx - v_0$ is a nonnegative linear combination of rays in $\{ u - v_0 \ : \ \{v_0, u\} \text{ is an edge of $P_0$ } \},$ there exists a vertex $v_1$ such that $\{v_0,v_1\}$ is an edge and 
	\[ \langle \bfa_F, v_0 \rangle < \langle \bfa_F, v_1 \rangle.\]
Continuing this procedure, we can construct a sequence of vertices of $P_0:$ $v_0=v, v_1, v_2, \dots, v_\ell$ such that $\{v_i, v_{i+1}\}$ is an edge of $P_0$ for each $i,$ and 
	\[ \langle \bfa_F, v \rangle < \langle \bfa_F, v_1 \rangle < \cdots < \langle \bfa_F, v_\ell \rangle = b_{0,F}.\]
	Using the assumption of (4), we get
	\[ \langle \bfa_F, v_\bb \rangle \le \langle \bfa_F, \left(v_1\right)_{\bb} \rangle \le \cdots \le \langle \bfa_F, \left(v_\ell\right)_{\bb} \rangle = b_{F},\]
	which is exactly the inequality $I_{v,F}(\bb)$ as desired.
%
%
\end{proof}

In this article, we will use the equivalence between (1) and (3) of Proposition \ref{prop:preredux} to determine the deformation cone $\Def(P_0).$ 
\begin{remark}
Part (3) of the proposition allows us to reduce the number of inequalities in determining the deformation cone. In the toric varieties language, this correspond to the fact that it is enough to check positivity on each torus invariant curve. See \cite[Theorem 6.3.12 part (c)]{cls}.
\end{remark}
It is undesirable to compute $v_\bb$ and then compute $\langle \bfa_{F}, v_\bb \rangle$ for each individual $E_{v,F}$ or $I_{v,F}$. We find the following explicit formulation useful. 

\begin{lemma}\label{lem:rewritelhs}
	Let $(v,F)$ be a vertex-facet pair of $P_0$. Suppose $F_{i_1}, F_{i_2}, \dots, F_{i_d}$ are the first $d$ supporting facets of $v$. If $\bfa_F = \sum_{j=1}^d c_j \bfa_{i_j} = \sum_{j =1}^d c_j \bfa_{F_{i_j}},$ then the left hand side of $E_{v,F}$ and $I_{v,F}$ becomes 
	$\displaystyle \sum_{j =1}^d c_j b_{{i_j}}$ or equivalently
	$\displaystyle \sum_{j =1}^d c_j b_{F_{i_j}}.$
\end{lemma}
\begin{proof}
$\displaystyle
	\langle \bfa_{F}, v_\bb \rangle = \left\langle \sum_{j=1}^d c_j\bfa_{F_{i_j}}, v_\bb \right\rangle
	 =\sum_{j=1}^d c_j \langle \bfa_{F_{i_j}}, v_\bb \rangle
	 = \sum_{j=1}^d c_j b_{F_{i_j}}.
$ 
\end{proof}


\subsection*{Deformation cones of simple polytopes} Finally, we apply our results to simple polytopes. We start with the following preliminary lemma.

\begin{lemma}\label{lem:edgeeq}
	Suppose $P_0$ is simple. Let $e = \{v, v'\}$ be an edge of $P_0.$ Suppose $F, F_{i_1},\dots, F_{i_{d-1}}$ are the supporting facets of $v,$ and $F', F_{i_1},\dots, F_{i_{d-1}}$ be the supporting facets of $v'.$ There is a unique solution $(c_F, c_{F'}, c_1, \dots, c_{d-1})$ up to scale to 
	\begin{equation}\label{equ:edgeeq}
		 \sum_{j=1}^{d-1} c_j \bfa_{F_{i_j}} = c_F \bfa_F + c_{F'} \bfa_{F'}
	\end{equation}
	such that $c_F c_{F'} > 0.$
Hence, there is a unique solution up to \emph{positive} scale to the above equation such that $c_F > 0, c_{F'} > 0.$
\end{lemma}

\begin{proof}
	The unique existence of a solution to \eqref{equ:edgeeq} such that $c_F c_F' \neq 0$ follow from the fact that both the set $\bfa_F, \bfa_{F_{i_1}},$ $\dots,$ $\bfa_{F_{i_{d-1}}}$ and the set $\bfa_{F'}, \bfa_{F_{i_1}}, \dots, \bfa_{F_{i_{d-1}}}$ are linearly independent. The numbers $c_F$ and $c_F'$ have the same sign because $\bfa_F$ and $\bfa_{F'}$ are on two different sides of the $(d-1)$-dimensional space spanned $\bfa_{F_{i_1}}, \dots, \bfa_{F_{i_{d-1}}}$ and $\sum_{j=1}^{d-1} c_j \bfa_{F_{i_j}}$ is a vector in this space.
\end{proof}
\begin{remark}
	Equation \eqref{equ:edgeeq} is called the \emph{wall condition} in \cite[Chapter 6]{cls}. See \cite[Figure 17, page 301]{cls}.
\end{remark}

\begin{definition}
	Assume all the hypotheses in Lemma \ref{lem:edgeeq} and let $(c_F, c_{F'}, c_1, \dots, c_{d-1})$ be the unique solution up to positive scale to \eqref{equ:edgeeq} assumed by Lemma \ref{lem:edgeeq}. (So $c_F, c_{F'} > 0$.) We associate to the edge $e=\{v, v'\}$ an inequality:
	\[ I_e(\bb): \sum_{j=1}^{d-1} c_j b_{F_{i_j}} \le c_F b_F + c_{F'}  b_{F'}.\]
\end{definition}

We now reach the main result of this part. 
\begin{proposition}\label{prop:redux} Suppose $P_0$ is as given in Setup \ref{setup1} and is simple. Let $\bb \in \RR^m.$ Then $\bb \in \Def(P_0)$ if and only if all the inequalities $I_{e}(\bb)$ are satisfied, where $e$ is an edge of $P_0.$ 
\end{proposition}

\begin{proof} We use the equivalence between (1) and (3) of Proposition \ref{prop:preredux}.
Since $P_0$ is simple. it is clear that condition (i) of Corollary \ref{cor:ineqs} can be ignored. Furthermore, any ordered pair of adjacent vertex $(v, v')$ of $P_0$ determines a unique neighboring vertex-facet pair $(v, F)$, where $F$ is the unique supporting facet of $v'$ that does not support $v,$ and any pair $(v,F)$ arises (not necessarily uniquely) this way. Therefore, we can change the indexing of the inequalities in condition (ii) of Corollary \ref{cor:ineqs} to $(v,v').$ 
Finally, one can verify if $e=\{v, v'\}$ is an edge, the inequality $I_e(\bb)$ is equivalent to both the inequality associated to $(v,v')$ and the one associated to $(v', v).$ Then the conclusion follows.
\end{proof}

\begin{example}\label{ex:exampleprop} 
We go back to our Example \ref{ex:example}, illustrated in Figure \ref{fig:exfigure}. We draw the polytope $P_0$ with a labeling of its vertices, and draw the normal fan $\Sigma(P_0)$ of $P_0$ in Figure \ref{fig:normalfan}.

\begin{figure}[t!]
\begin{tikzpicture}
	
	\begin{scope}[xshift=7cm, scale=0.8]

\draw[thick,->] (0,0) -- (-1,0);
\node[left] at (-1,0) {$\scriptstyle \bfa_1=(-1,0)$};
\draw[thick,->] (0,0) -- (0,1);
\node[above] at (0,1) {$\scriptstyle \bfa_2=(0,1)$};
\draw[thick,->] (0,0) -- (0,-1);
\node[below] at (0,-1) {$\scriptstyle \bfa_3=(0,-1)$};
\draw[thick,->] (0,0) -- (1,-1);
\node[above right] at (0.8,-0.8) {$\scriptstyle \bfa_4=(1,-1)$};
\node at (-3,1){$\textcolor{blue}{\Sigma(P_0)}:$};
\end{scope}

\begin{scope}[xshift=-5cm, scale=0.9]
\draw (6,0.5)--(8,0.5)--(7,-0.5)--(6,-0.5)--cycle;
\node[above] at (6,0.5) {$v$};
\node[above] at (8,0.5){$w$};
\node[below] at (7,-0.5) {$x$};
\node[below] at (6,-0.5) {$y$};
\node at (6.7,0){$\textcolor{blue}{P_0}$};
\draw[fill] (6,-0.5) circle [radius=0.08];
\draw[fill] (6,0.5) circle [radius=0.08];
\draw[fill] (7,-0.5) circle [radius=0.08];
\draw[fill] (8,0.5) circle [radius=0.08];
\end{scope}
\end{tikzpicture}
\caption{Normal fan of polytope $P_0$ in Example \ref{ex:example}.}\label{fig:normalfan}
\end{figure}
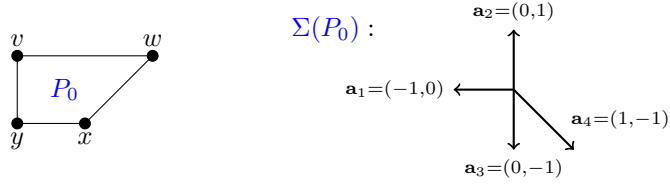
Now we apply Proposition \ref{prop:redux} to find the inequalities that define $\Def(P_0).$ 
	Let $e_1 = \{v, y\}.$ The vertex $v$ lies on facets $F_1$ and $F_2$, and the vertex $y$ lies on facets $F_1$ and $F_3.$ We have $0 \cdot \bfa_1=\bfa_2 + \bfa_3.$ This gives the inequality $I_{e_1}: 0 \le b_2 + b_3.$
Similarly, for $e_2 := \{v, w\},$ we have $-\bfa_2 =\bfa_1 + \bfa_4,$ which gives $I_{e_2}: - b_2 \le b_1 + b_4;$ for $e_3 = \{y,x\},$ we have $\bfa_3=\bfa_1 + \bfa_4,$ which gives $I_{e_3}: b_3 \le b_1 + b_4;$ for $e_4 = \{x,w\},$ we have $0 \cdot \bfa_4=\bfa_2 + \bfa_3 ,$ which gives $I_{e_4}: 0 \le b_2 + b_3.$

Note that two of the four inequalities $I_{e_1}$ and $I_{e_4}$ are the same, and $I_{e_2}$ follows from $I_{e_1}$ and $I_{e_3},$ so is redundant. Therefore, $\Def(P_0)$ is defined by two inequalities in $\RR^4:$
\begin{equation}
I_{e_1}=I_{e_4}: 0 \le b_2+ b_3, \qquad I_{e_3}: b_3 \le b_1 + b_4.
	\label{equ:exdefcone}
\end{equation}
Among the three vectors given in Example \ref{ex:example}, we can verify that $\bb_1=(3,2,0,2.5), \bb_2=(1,2,1,0)$ satisfy the above two inequalities, and $\bb_3=(1,2,2,0)$ does not satisfy the inequality $I_{e_3}.$ This agrees with the assertion that $\bb_1, \bb_2 \in \Def(P_0)$ and $\bb_3 \not\in \Def(P_0).$

We remark that $\Def(P_0)$ defined by \eqref{equ:exdefcone} is not pointed. Indeed, for any deformation $Q$ of $P_0,$ any translation of $Q$ is also a deformation of $P_0.$ We may consider two polytopes are equivalent if one is obtained from another by translation. Under this equivalence, the collection of the deforming vectors gives the nef cone $\Nef(P_0)$ of $P_0$. (See Remark \ref{rem:nef}.) 
One sees that $\Nef(P_0)$ is computed from $\Def(P_0)$ by quotienting out the span of the two columns of $\A$ 
which are $(-1,0,0,1)^T,(0,1,-1,-1)^T.$
In this quotient we write everything in terms of $b_3,b_4$ since we have $b_1=b_4$ and $b_2=b_3+b_4$. So the nef cone $\Nef(P_0)$ is defined by
\[ 0 \leq  2b_3+b_4, \qquad  0 \leq  -b_3+2b_4, \]
 where $b_3,b_4$ are the coordinates of $\RR^2$. This is always a pointed cone.
\end{example}


\subsection*{Deformation cones of simplicial projective fans}
A fan $\Sigma$ is \emph{simplicial} if every cone in it is simplicial. This means that every $k$-dimensional cone in $\Sigma$ is spanned by exactly $k$ rays. One sees that $P$ being simple is equivalent to that $\Sigma(P)$ is simplicial. In particular, edges of $P$ are in bijection with a pair of adjacent maximal cones in $\Sigma(P),$ where we say two maximal cones are \emph{adjacent} if their spanning ray sets differ by exactly one ray. 
We can easily translate Lemma \ref{lem:edgeeq} and Proposition \ref{prop:redux} to versions for simplicial fans using the connection between a simple polytope and its simplicial normal fan. We omit the modified version of Lemma \ref{lem:edgeeq}, but restate Proposition \ref{prop:redux} since the new version will be the main one we use in Sections \ref{sec:GP} and \ref{sec:NBF}. 
\begin{definition}\label{defn:fanineq}
	Suppose $\Sigma_0$ is simplicial. Let $\left\{\bfa_F, \bfa_{F_{i_1}},\dots, \bfa_{F_{i_{d-1}}} \right\}$ and $\Big\{ \bfa_{F'}, \bfa_{F_{i_1}},$ $\dots,$ $\bfa_{F_{i_{d-1}}} \Big\}$ be the sets of spanning rays of two adjacent maximal cones $\sigma$ and $\sigma'$ in $\Sigma_0.$ 
Suppose $(c_F, c_{F'}, c_1, \dots, c_{d-1})$ is the unique solution up to positive scale to \eqref{equ:edgeeq} assumed by Lemma \ref{lem:edgeeq}. (So $c_F, c_{F'} > 0$.) We associate to the pair $\{\sigma, \sigma'\}$ an inequality:
\[ I_{ \{\sigma,\sigma'\}} (\bb): \sum_{j=1}^{d-1} c_j b_{F_{i_j}} \le c_F b_F + c_{F'}  b_{F'}.\]	
\end{definition}
\begin{proposition}\label{prop:reduxfan} Suppose $\Sigma_0$ is as given in Setup \ref{setup2} and is simplicial. Let $\bb \in \RR^m.$ Then $\bb \in \Def(\Sigma_0)$ if and only if all the inequalities $I_{ \{\sigma, \sigma'\}}(\bb)$ are satisfied, where $\{ \sigma, \sigma'\}$ is a pair of adjacent maximal cones in $\Sigma_0.$ 
\end{proposition}

\subsection*{Back to Deformations}
We finish this section with a discussion on how to determine whether a polytope $Q$ is a deformation of $P_0$ provided that we have a description for the deformation cone $\Def(P_0).$ Although there is a one-to-one correspondence between deforming vectors $\bb \in \Def(P_0)$ and deformations of $P_0,$ if we take a polytope $Q$ that is defined by $\A \xx \le \bb$, knowing $\bb \not\in \Def(P_0)$ is not enough to conclude that $Q$ is not a deformation of $P_0.$ Indeed, we have seen in Examples \ref{ex:example} and \ref{ex:exampleprop} that $Q_3$ (in Figure \ref{fig:exfigure}) is defined by $\A \xx \le \bb_3$ where $\bb_3 \not\in \Def(P_0);$ but $Q_3$ is a deformation of $P_0$. 
It was discussed earlier that the reason for which $\bb_3$ is not a deforming vector is that the hyperplane defined by $-y=2$, i.e., the bottom horizontal line in the picture for $Q_3$, does not ``touch'' the polytope $Q_3.$ 
This turns out to be an important notion.
\begin{definition}
Suppose a polytope $Q \subset V$ is defined by the linear system $\A \xx \le \bb.$ We say an inequality $\langle \bfa_i, \xx \rangle \le b_i$ in the system is \emph{tight} for $Q$, if the equality is attained for some points in $Q.$ If all the inequalities in the system are tight, we say $\A \xx \le \bb$ is a \emph{tight} representation for $Q.$ 
\end{definition}
It is easy to see that $\A \xx \le \bb$ being a tight representation for $Q$ is a consequence of condition \ref{item:nopass} of Definition \ref{defn:deform0}, and thus is a necessary condition for $\bb$ being a deforming vector. With this concept of tight representations, we can use the knowledge of the deformation cone to verify whether a polytope $Q$ is a deformation of $P_0.$

\begin{lemma}\label{lem:checkdeform}
	Suppose $P_0$ is as described in Setup \ref{setup1}, and $Q$ is defined by a tight representation $\A \xx \le \bb.$ Then $Q$ is a deformation of $P_0$ if and only if $\bb \in \Def(P_0).$
\end{lemma}

\begin{proof}
We only need to show the forward implication as the backward one is obvious. Suppose $Q$ is a deformation of $P_0.$ Then there exists $\bb' \in \Def(P_0)$ such that $Q$ is defined by $\A \xx \in \bb',$ which is a tight representation as well. By the definition of tightness, we have
\[ b_i = \max_{\xx \in Q} \langle \bfa_i, \xx \rangle = b_i', \quad \forall 1 \le i \le m. \]
Hence, $\bb = \bb' \in \Def(P_0).$
\end{proof}

\section{Generalized permutohedra and Braid Fan} \label{sec:GP}

In the following two sections, we work over the vector space $V_d = \{ \xx \in \RR^{d+1} \ : \langle\1, \xx\rangle = 0\} \subset \RR^{d+1}$ and its dual space $W_d = \RR^{d+1}/\1$, where $\1 = (1, 1, \dots, 1)$ denotes the all-one vector in $\RR^{d+1}.$ Note that the standard basis $\{\ee_1,\cdots,\ee_{d+1}\}$ of $\RR^{d+1}$ is a canonical spanning set for $W_d$ although it is not a basis.

The goal of this section is to apply the techniques introduced in Section \ref{sec:prel} to give a new combinatorial proof for Theorem \ref{thm:submodular} by determining the deformation cone of the Braid fan. We also state and prove Theorem \ref{thm:polymatroid}, which gives the connection between polymatroids and generalized permutohedra. We start by introducing the fan concerned in this section.

\begin{definition}
For any $\pi\in \fS_{d+1}$ we define a cone in $W_d$ as follows:
\[
C(\pi):=\{\xx\in W_d \ :\ x_{\pi^{-1}(1)}<x_{\pi^{-1}(2)}<\cdots<x_{\pi^{-1}(d+1)}\}.
\]
\end{definition}

One checks that $C(\pi)$ is well-defined because if $(x_1, \dots, x_{d+1}) = (y_1, \dots, y_{d+1})$ is in $W_d$, that is, there exists $k \in \RR$ such that $y_i = x_i +k$ for each $i,$ then 
\[ x_{\pi^{-1}(1)}<x_{\pi^{-1}(2)}<\cdots<x_{\pi^{-1}(d+1)} \text{ if and only if } y_{\pi^{-1}(1)} <y_{\pi^{-1}(2)}  <\cdots<y_{\pi^{-1}(d+1)}.\]
Also, for any two distinct $\pi_1,\pi_2 \in \fS_{d+1},$ the cones $C(\pi_1)$ and $C(\pi_2)$ are disjoint.
Each region $C(\pi)$ is an open polyhedral cone. Its closure, denoted by $\sigma(\pi)$, is obtained from $C(\pi)$ by relaxing the strict inequalities.

\begin{definition}
	We call the collection of cones $\{\sigma(\pi): \pi\in \fS_{d+1}\}$, together with all of their faces, the \emph{Braid fan}, denoted by $\Br_d.$
\end{definition}
It is straightforward to show that $\Br_d$ is a complete fan in $W_d.$ However, we will prove this fact by showing $\Br_d$ is the normal fan of a family of polytopes in Proposition \ref{prop:fanofusual} below. The similar idea will be used in the next section, and we simply present it this way in preparation for later discussions. 

We next formally introduce generalized permutohedra.
Given a strictly increasing sequence $\balpha= (\alpha_1,\alpha_2,\cdots,\alpha_{d+1}) \in \RR^{d+1}$, for any $\pi\in \fS_{d+1}$, we use the following notation: 
\[ v_\pi^\balpha := \left(\alpha_{\pi(1)},\alpha_{\pi(2)},\cdots, \alpha_{\pi({d+1})}\right) = \sum_{i=1}^{d+1} \alpha_{i} \ee_{\pi^{-1}(i)}.\]
Then we define the \emph{usual permutohedron}
\begin{equation*}
\Perm(\balpha) := 
\conv\left(v_\pi^{\balpha}:\quad \pi\in \fS_{d+1}\right).
\label{equ:defnusual}
\end{equation*}
In particular, if $\balpha = (1, 2, \dots, {d+1}),$ we obtain the \emph{regular permutohedron}, denoted by $\Pi_{d},$
\[ \Pi_{d} := \Perm (1, 2, \dots, d+1).\]
Note that the above definition for $\Perm(\balpha)$ does not directly say that the vertex set of $\Perm(\balpha)$ is $\{ v_\pi^\balpha : \ \pi \in \fS_{d+1}\}$. However, this is true as we will see in Proposition \ref{prop:fanofusual} below.

Recall that \emph{generalized permutohedra} are polytopes obtained from usual permutohedra by moving vertices while preserving all edge directions. 
We see that any generalized permutohedron in $\RR^{d+1}$ lies in an affine space that is parallel to $V_d.$ However, under the setup of our article, we would like to only consider polytopes that are in $V_d.$ Thus, we give the following definition. 
\begin{definition}\label{defn:central}
	Suppose $V'$ is an affine space that is parallel to $V_d,$ and $P \in \RR^{d+1}$ is a polytope that lies in $V'$. It is clear $V' =\{ \xx \in \RR^{d+1} : \langle \1, \xx \rangle = N\}$ for some (unique) $N \in \RR.$ Then $V' = \frac{N}{d+1} \1 + V_d.$ We define $\widetilde{P} := P - \frac{N}{d+1} \1$ to be the \emph{centralized} version of the polytope $P,$ which lies in $V_d.$  
\end{definition}

\begin{example}\label{ex:crp}
	The regular permutohedron $\Pi_d$ lies in the affine space $V' = \Big\{ \xx \in \RR^{d+1} \ : \ \langle \1, \xx \rangle = \frac{(d+2)(d+1)}{2}\Big\}.$ Hence, the \emph{centralized regular permutohedron} is
	\[ \widetilde{\Pi_d} = \Pi_d - \frac{d+2}{2} \1 = \Perm\left( -\frac{d}{2}, -\frac{d-2}{2}, \dots, \frac{d-2}{2}, \frac{d}{2} \right) \subset V_d.\]
\end{example}

We have the following two results relating generalized/usual permutohedra and $\Br_d.$
\begin{proposition}[Proposition 2.6 in \cite{post}]\label{prop:fanofusual}
	If $\balpha=(\alpha_1,\dots, \alpha_{d+1})$ is strictly increasing, 
	then for each $\pi \in \fS_{d+1},$ the point $v_\pi^\balpha$ is a vertex of $\Perm(\balpha),$ and the normal cone of $\Perm(\balpha)$ at $v_{\pi}^\balpha$ is $\sigma(\pi).$ Therefore, 
	the Braid fan $\Br_d$ is the normal fan of the usual permutohedron $\Perm(\balpha).$ Hence, $\Br_d$ is a complete projective fan in $W_d$.
\end{proposition}

\begin{proposition}[Proposition 3.2 in \cite{PosReiWil}] \label{prop:coarser}
  A polytope $P$ in $V_d$ is a (centralized) generalized permutohedron if and only if its normal fan $\Sigma(P)$ is refined by the braid arrangement fan $\Br_d$.
\end{proposition}

We include a proof for Proposition \ref{prop:fanofusual}, which is relevant to discussion in Section \ref{sec:NBF}. The following elementary result is useful (See \cite[Theorem 368]{inequalities}).

\begin{lemma}[Rearrangement Inequality]\label{lem:rearr} 
Suppose $x_1 \le x_2 \le \cdots \le x_n$ and $y_1 \le y_2 \le \cdots \le y_n.$ Then for any $\pi \in \fS_n,$ we have
\[ \sum_{i=1}^n x_i y_i \ge \sum_{i=1}^n x_i y_{\pi(i)}.\]
Furthermore, if $x_1 < x_2 < \cdots < x_n$ and $y_1 < y_2 < \cdots < y_n$, then the equality only holds when $\pi$ is the identity permutation.
\end{lemma}

Recall the definition of normal cone and the notation $\ncone(F,P)$ in Definition \ref{defn:normal}

\begin{proof}[Proof of Proposition \ref{prop:fanofusual}]
Let $\ww \in C(\pi).$ For convenience, 
we let $u_i = w_{\pi^{-1}(i)}$ so that $\ww$ can be expressed as 
\begin{equation}\label{equ:wexp0}
	\ww = \sum_{i=1}^{d+1} u_{i}\ee_{\pi^{-1}(i)}.
\end{equation}
Then $\ww \in C(\pi)$ means that 
\[ u_{1} < u_{2} < \cdots < u_{d+1}.\]
Then it follows from Lemma \ref{lem:rearr} that $\langle \ww, v_\pi^\balpha \rangle > \langle \ww, v_{\pi'}^\balpha \rangle,$ for any $\pi \neq \pi' \in \fS_{d+1}.$
Hence, $v_\pi^\balpha$ does not lie in $\conv( v_{\pi'}^\balpha : \ \pi \neq \pi' \in \fS_{d+1});$ so $v_\pi^\balpha$ is a vertex of $\Perm(\balpha)$. Furthermore, we must have that $\ww \in \ncone(v_\pi^\balpha, \Perm(\balpha)).$ This implies 
\begin{equation}
\label{equ:inclusion1}
\sigma(\pi) \subseteq \ncone(v_\pi^\balpha, \Perm(\balpha)).
\end{equation}
However, the union of $\sigma(\pi)$ is the entire space $W_d,$ so the equality must holds in \eqref{equ:inclusion1}. Thus, the conclusion follows. 
\end{proof}

It follows from Propositions \ref{prop:fanofusual} and \ref{prop:coarser} that the deformation cone $\Def(\Br_d)$ of $\Br_d$ is the same as the deformation cone $\Def\left( \widetilde{\Pi_d} \right)$ of $\widetilde{\Pi_d}$, which gives a characterization for (centralized) generalized permutohedron.

The combinatorics of the Braid fan $\Br_d$ are equivalent to those of the face lattice of $\widetilde{\Pi_d}$, which are well-studied in the literature \cite[Chapter VI, Proposition 2.2]{barvinokconvex}. 
We summarize relevant results in terms of the Braid fan $\Br_d$ in the proposition below.
Recall that the \emph{Boolean algebra} $\BB_{d+1}$ is the poset on all subsets of $[d+1]$ ordered by containment. This poset has a minimum element $\hat{0}=\emptyset$ and a maximum element $\hat{1}=[d+1]$. 
We denote by $\overline{\BB_{d+1}}$ the poset obtained from $\BB_{d+1}$ by removing the maximum and minimum elements. 
For each element $S  \in \BB_{d+1}$, define
\[
	\displaystyle \ee_S := \sum_{i \in S} \ee_i.
\]

\begin{proposition}\label{prop:charbr}
The rays, i.e., $1$-dimensional cones, of the Braid fan $\Br_d$ are given by $\ee_S$ for all $S \in \overline{\BB_{d+1}}$.  
Furthermore, a $k$-set of rays $\{\ee_{S_1},\cdots, \ee_{S_k}\}$ spans a $k$-dimensional cone in $\Br_d$ if and only if the sets $S_1,\dots, S_k$ form a $k$-chain in $\overline{\BB_{d+1}}$.

In particular, the maximal cones in $\Br_d$ are in bijection with the maximal chains in $\overline{\BB_{d+1}}.$ Hence, $\Br_d$ is simplicial. 

\end{proposition}

As the one-dimensional cones are indexed by elements in $\overline{\BB_{d+1}},$ the deformation cone of $\Br_d$ can be considered to be in $\RR^{\overline{\BB_{d+1}}}$ which is indexed by nonempty, proper subsets $S$ of $[d+1].$

With these results in hand, we can now apply Proposition \ref{prop:reduxfan} to compute $\Def(\Br_d)$ .

\begin{theorem}\label{thm:centralsub}
	The deformation cone of the Braid fan (or centralized regular permutohedron) is the collection of $\bb \in \RR^{\overline{\BB_{d+1}}}$ satisfying the following \emph{submodular condition} on $\BB_{d+1}:$
\begin{equation}\label{equ:bsub}
b_{S \cup T} + b_{S \cap T} \le b_S + b_T, \quad \forall S, T \in {\BB_{d+1}},
\end{equation}
where by convention we let $b_\emptyset = b_{[d+1]} = 0.$
\end{theorem}

\begin{figure}[h]
\begin{tikzpicture}
	\begin{scope}[scale=0.8]
    \draw [black,fill] (0,0) circle [radius = 0.08];
    \draw[black, fill] (0,-1) circle [radius=0.08];
    \draw[black, fill] (1,1) circle [radius=0.08];
    \draw[black, fill] (-1,1) circle [radius=0.08];
    \draw[black, fill] (0,2) circle [radius=0.08];
    \draw[black, fill] (0,3) circle [radius=0.08];
    \draw (0,-1) -- (0,0) -- (1,1) -- (0,2);
    \draw (0,0) -- (-1,1) -- (0,2)--(0,3);
    \node[above] at (0,3.2) {\vdots};
    \node[below] at (0,-1) {\vdots};
    \node[below right] at (0,-1) {$S_{i-2}$};
    \node[below right] at (0,0) {$S_{i-1}$};
    \node[right] at (1,1) {$S'_{i}$};
    \node[left] at (-1,1) {$S_{i}$};
    \node[right] at (0,2) {$S_{i+1}$};
    \node[right] at (0,3) {$S_{i+2}$};
\end{scope}
\end{tikzpicture}
\caption{A diamond in the Boolean algebra.}
\label{fig:poset}
\end{figure}
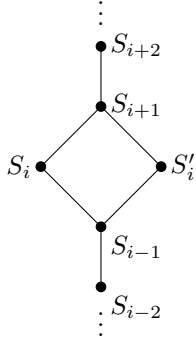

\begin{proof}
We may add $\emptyset$ and $[d+1]$ back to $\overline{\BB_{d+1}}$ and say that the maximal cones in $\Br_d$ are in bijection with maximal chains in $\BB_{d+1}.$
Then any pair of adjacent maximal cones in $\Br_d$ corresponds to a pair of maximal chains in $\BB_{d+1}$ that only differ at a non-extreme element, and all pairs of adjacent maximal cones arise this way. 
One sees any such pair of maximal chains in $\BB_{d+1}$ always form a ``diamond'' shape as shown in Figure \ref{fig:poset}. Suppose we have a pair of maximal chains shown in Figure \ref{fig:poset}. Then if we let $a = S_i \setminus S_{i-1}$ and $b = S_{i}' \setminus S_{i-1},$ we must have that $S_{i+1} = S_i \cup \{a, b\}.$ Therefore,
\begin{equation}\label{equ:eesum}
\ee_{S_{i+1}} + \ee_{S_{i-1}} = \ee_{S_{i}} + \ee_{S_i'},
\end{equation}
which is precisely the solution to \eqref{equ:edgeeq} assumed by Lemma \ref{lem:edgeeq}. (Note that if $i=1,$ then $\ee_{S_{i-1}} = \ee_\emptyset=0$, and if $i=d,$ then $\ee_{S_{i+1}} = \ee_{[d+1]} = \1 = 0$ in $W_d.$ For both cases, \eqref{equ:eesum} is the expression that we need.) It follows from Proposition \ref{prop:reduxfan} that the corresponding pair of adjacent maximal cones gives us the following inequality:
\[ b_{S_{i+1}} + b_{S_{i-1}} \le b_{S_{i}} + b_{S_i'}.\]
Going through all pairs of adjacent maximal cones, we see that $\Def(\Br_d)$ is defined by the following collection of inequalities:
\begin{equation}\label{equ:bsubdiamond}
	b_{S \cup \{a,b\}} + b_{S} \le b_{S \cup \{a\}} + b_{S \cup \{b\}}, \text{ for all $S \subseteq [d+1]$ and $a, b \in [d+1]\setminus S.$}
\end{equation}

Finally, we show that each inequality given by \eqref{equ:bsub} follows from the above set of inequalities by induction on the size difference between $S \cup T$ and $S \cap T$.
If $|S \cup T| - |S \cap T| = 0$ or $1$, one checks that the equality in \eqref{equ:bsub} holds. If $|S \cup T| - |S \cap T| = 2,$ then \eqref{equ:bsub} becomes an inequality in the form of \eqref{equ:bsubdiamond}. Now let $n > 2$ and assume that \eqref{equ:bsub} holds for any $S, T \in \BB_{d+1}$ satisfying $|S \cup T| - |S \cap T| < n.$ Suppose $S, T \in \BB_{d+1}$ satisfying $|S \cup T| - |S \cap T| = n.$ Let 
\[ R := S \cap T, \quad A:= S \setminus R, \quad B := T \setminus R.\]
Then $R, A, B$ are pairwise disjoint, and
\[ S = R \cup A, \quad T = R \cup B, \quad S \cup T = R \cup A \cup B, \quad |A|+|B| = n.\]
 Without loss of generality, we assume $|A| \ge |B|.$ Since $n \ge 2,$ we have that $|A| \ge 1.$ Pick $a \in A.$ Note that 
 $|R \cup A \cup B| - |R \cup \{a\}| = n -1 < n,$ and because $|B| \le |A|$,
 \[ |R \cup B \cup \{a\}| - |R| = |B| + 1 \le n/2 + 1 < n.\]
 Hence, by the induction hypothesis, we have
\begin{align*}
	b_{S \cup T} + b_{R \cup \{a\}} = b_{R \cup A \cup B} + b_{R \cup \{a\}} \le & b_{R \cup A} + b_{R \cup B \cup \{a\}} = b_S + b_{T \cup \{a\}}, \\
	b_{T \cup \{a\}} + b_{S \cap T} = b_{R \cup B \cup \{a\}} + b_{R} \le & b_{R \cup \{a\}} + b_{R \cup B} = b_{R \cup \{a\}} + b_{T}.
\end{align*}
Adding these two inequalities, we obtain \eqref{equ:bsub}, completing the proof.
\end{proof}

\begin{remark}\label{rem:equivsub}
	We see from the proof of Theorem \ref{thm:centralsub} that the submodular condition \eqref{equ:bsub} is equivalent to the ``diamond'' submodular condition \eqref{equ:bsubdiamond}. 
	This equivalence will be used again in Section \ref{sec:NBF}.
\end{remark}

	Note that points in $\RR^{\BB_{d+1}}$ can be considered as set functions from $2^{[d+1]}$ to $\RR.$ We now restate the Submodularity Theorem with more details and prove it. 
	
	\begin{theorem}[Submodularity Theorem, restated] \label{thm:submodrestate}
	For each submodular function $\bb \in \RR^{\BB_{d+1}}$ satisfying $\bb_\emptyset=0,$ the linear system:
	\begin{equation} \label{equ:linear}
		\left\langle \ee_{[d+1]}, \xx \right\rangle = \left\langle \1, \xx \right\rangle \ = \ b_{[d+1]}, \quad \text{and} \quad 
		\left\langle \ee_S, \xx \right\rangle \ \le \ b_S, \quad \forall \emptyset \neq S \subsetneq [d+1]
	\end{equation}
	defines a generalized permutohedron in $\RR^{d+1},$ and any generalized permutohedron arises this way uniquely.

	Furthermore, if a polytope $P \in \RR^{d+1}$ is defined by a tight representation \eqref{equ:linear}, then $P$ is a generalized permutohedron if and only if $\bb$ is a submodular function $\bb \in \RR^{\BB_{d+1}}$ satisfying $\bb_\emptyset=0.$
\end{theorem}

\begin{proof}
	It follows from Theorem \ref{thm:centralsub}, Proposition \ref{prop:coarser} and the description for rays of $\Br_d$ in Proposition \ref{prop:charbr} that the one-to-one correspondence described by the theorem holds for centralized generalized permutohedra and submodular functions $\bb \in \RR^{\BB_{d+1}}$ satisfying $\bb_\emptyset=0$ and $\bb_{[d+1]} = 0.$ 
	
	Suppose $\bb \in \RR^{\BB_{d+1}}$ is a set function. Let $k = \frac{b_{[d+1]}}{d+1}$ and define a new vector/function $\bb'$ by
	\[ \bb'_S = \bb_S - k |S|, \quad \forall S \subseteq [d+1].\]
	Let $P$ and $Q$ be the polytopes defined by the linear system \eqref{equ:linear} with vectors $\bb$ and $\bb'$ respectively. It is straightforward to check the following facts are true:
\begin{enumerate}
	\item $\bb'_\emptyset = \bb_\emptyset$ and $\bb'_{[d+1]} =0.$
	\item $\bb'$ is a submodular function if and only if $\bb$ is a submodular function.
	\item $Q =\tilde{P} = P - k \1$ is the centralized version of $P.$
\end{enumerate}
The first conclusion of the theorem follows from these facts and the arguments in the first paragraph.

Finally, the second conclusion follows from Lemma \ref{lem:checkdeform} and the observation that $\A \xx \le \bb$ is a tight representation for $P$ if and only if $\A \xx \le \bb'$ is a tight representation for $Q.$
	\end{proof}

	\begin{remark}\label{rem:edges}
		We remark that other than the Submodularity Theorem and Proposition \ref{prop:coarser}, there is another characterization of generalized permutohedra in terms of edges. A polytope $P \in \RR^{d+1}$ is a generalized permutohedron if and only if all of its edge directions are in the form of $\ee_i-\ee_j$ for $1 \le i < j \le d+1$.
		We briefly give the proof for the forward implication of the above statement, which will be used in the example we discuss below. 
		Indeed, it follows from Proposition \ref{prop:coarser} that 
		for each cone $\sigma$ of codimension $1$ in the normal fan $\Sigma(P)$ of a generalized permutohedron $P$, there exists a cone $\sigma'$ of codimension $1$ in $\Br_d$, such that the $(d-1)$-dimensional linear space spanned by $\sigma$ is the same of the linear space spanned by $\sigma',$ and hence the direction of the edge associated with $\sigma$ in $P$ has the same direction as the direction of the edge associated with $\sigma'$ in the regular permutohedron $\Pi_d.$ It is straightforward to verify that all the edge directions of $\Pi_d$ are in the form of $\ee_i-\ee_j.$  
\end{remark}

	\begin{example}\label{ex:unexample} 
		Consider the polytope $P$ in $\RR^4$ defined by the linear system \eqref{equ:linear} with $b_{[4]} = 6$ and $b_S =3$ if $|S|=1$, $b_S=4$ if $|S|=2$, and $b_S = 6$ if $|S|=3.$
%
We see that 
\[
	8=4 + 4 = b_{\{1,2\}} + b_{\{2,3\}} < b_{\{1,2,3\}} + b_{\{2\}} = 6 + 3 = 9.
\]
So $\bb$ is not a submodular function. Since the given system is a tight representation for $P$, we conclude that $P$ is a not a generalized permutohedron. Indeed, $P$ is the cube whose vertices are $(1,1,1,3),(0,2,2,2)$ and their permutations. The linear functional given by the vector $(1,2,3,4)$ attains its maximum at the vertices $(0,2,2,2)$ and $(1,1,1,3)$, but not at the other vertices. Thus, $(0,2,2,2)$ and $(1,1,1,3)$ form an edge whose direction is parallel to $(-1,1,1,-1)$, conflicting with the condition for being a generalized permutohedron expressed in Remark \ref{rem:edges}. 
\end{example}

\subsection*{Polymatroids vs Generalized Permutohedra.}\label{subsec:polyvsperm}
We finish this section by making the connection between polymatroids and generalized permutohedra.

\begin{definition}\label{defn:rank}
A \emph{polymatroid rank function} 
is a set function $r: 2^E \to \RR$ on a finite set $E$ such that 
\begin{itemize}
\item[(R1)] $0\leq r(A)$ for all $A\subseteq E$. (Nonnegativity condition)
\item[(R2)] If $A_1\subseteq A_2\subseteq E$ then $r(A_1)\leq r(A_2)$. (Monotone condition)
\item[(R3)] $r(A_1\cup A_2) + r(A_1\cap A_2)\leq r(A_1) + r(A_2)$ for all $A_1,A_2\subseteq E$. (Submodular condition)
\end{itemize}
\end{definition}
Note that we are only lifting the restriction $r(A)\leq |A|$ from the definition of matroids. 
To be consistent with notation used for generalized permutohedra, we may assume $E = [d+1]$, and $r = \bb \in \RR^{\BB_{d+1}}.$
\begin{definition}
	The \emph{base polymatroid} $P_\bb$ associated to a polymatroid rank function $\bb$ on $[d+1]$ is the polytope in $\RR^{d+1}$ defined by the linear system \eqref{equ:linear}. 
\end{definition}

It turns out that we may add the constraint $\bb_\emptyset = 0$ to the above definition and still get all the base polymatroids.
\begin{lemma}\label{lem:reduce20}
	Let $\bb \in \RR^{\BB_{d+1}}$ be a polymatroid rank function. Define an associated vector $\bb'$ as follows:
	\[ \bb'_\emptyset = 0, \quad \text{and} \quad \bb'_S = \bb_S, \forall \emptyset \neq S \subseteq [d+1].\]
	Then $\bb'$ is a polymatroid rank function on $[d+1]$ and $P_\bb = P_{\bb'}.$
\end{lemma}
The proof of the lemma is straightforward, so is omitted.

\begin{theorem}\label{thm:polymatroid}
The bijection asserted in Theorem \ref{thm:submodrestate} induces a bijection between base polymatroids of dimension at most $d$ and monotone submodular functions $\bb \in \RR^{\BB_{d+1}}$ satisfying $\bb_\emptyset=0.$

Moreover, every generalized permutohedron has a translation that is a base polymatroid.
\end{theorem}

\begin{proof}
	The first assertion follows easily from Lemma \ref{lem:reduce20}, Theorem \ref{thm:submodrestate}, and the observation that the nonnegativity condition (R1) follows from the monotone condition when we assume $\bb_\emptyset=0.$

	We use similar ideas presented in the proof of Theorem \ref{thm:submodrestate} to prove the second statement. Suppose $P$ is a generalized permutohedron associated to the submodular function $\bb \in \RR^{\BB_{d+1}}$ (where $\bb_\emptyset = 0$). For any $k \in \RR$, we define a new vector/function $\bb^{(k)} \in \RR^{\BB_{d+1}}$ by
	\[ \bb^{(k)}_S = \bb_S + k |S|, \quad \forall S \subseteq [d+1].\]
	Then $\bb^{(k)}$ is a submodular function and the generalized permutohedron associated to $\bb^{(k)}$ is a translation of $P.$ However, it is easy to see that for sufficiently large $k,$ the set function $\bb^{(k)}$ is monotone. Hence, the conclusion follows.
\end{proof}

\section{Nested Braid fan and nested permutohedra} \label{sec:NBF}


The plan of this section is as follows: We will first introduce the nested Braid fan $\Br_d^2$ as a refinement of the Braid fan, and construct a family of polytopes, called usual nested permutohedra in $V_d$ by giving an explicit description for their vertices. We then establish (in Proposition \ref{prop:fanofusual2}) the connection between these two new objects by showing $\Br_d^2$ is the normal fan of any usual nested permutohedron. After discussing combinatorial structure of $\Br_d^2$ (in Proposition \ref{prop:charbr2}), we give an inequality description for usual nested permutohedra (in Theorem \ref{thm:facetdes}). Lastly, we determine deformation cones of $\Br_d^2$ and nested permutohedra and give a result that is analogous to the Submodularity Theorem (see Theorems \ref{thm:centralsub2} and \ref{thm:sub2}). 

Recall that $\{\ee_1,\cdots,\ee_{d+1}\}$ is the standard basis for $\mathbb{R}^{d+1}$. For any permutation $\pi \in \fS_{d+1},$ we define
\[ \ff^\pi_i := \ee_{\pi^{-1}(i+1)} - \ee_{\pi^{-1}(i)}, \quad \forall 1 \le i \le d,\]
and for any point $\xx =(x_1,\dots,x_{d+1})$ in $\RR^{d+1}$ or $W_d$, we define
\[ (\Delta \xx)^\pi_i := x_{\pi^{-1}(i+1)} - x_{\pi^{-1}(i)}, \quad \forall 1 \le i \le d.\]

\begin{definition}\label{defn:NBF}
For each $(\pi,\tau)\in \fS_{d+1}\times \fS_{d}$, let $C(\pi,\tau)$ be the collection of vectors $\xx \in W_d$ satisfying:
\begin{enumerate}
\item $x_{\pi^{-1}(1)}<x_{\pi^{-1}(2)}<\cdots<x_{\pi^{-1}(d+1)}$, and
\item $(\Delta \xx)^\pi_{\tau^{-1}(1)}<(\Delta\xx)^\pi_{\tau^{-1}(2)}<\cdots < (\Delta\xx)^\pi_{\tau^{-1}(d)}$. (Note that this condition is an order of the first differences of the sequence $x_{\pi^{-1}(1)}, x_{\pi^{-1}(2)}, \dots, x_{\pi^{-1}(d+1)}$ with respect to the permutation $\tau.$) 
\end{enumerate}
\end{definition}
Similar to $C(\pi)$ defined in the last section, one can check that $C(\pi,\tau)$ is well-defined, and each region $C(\pi,\tau)$ is an open polyehdral cone. Let $\sigma(\pi,\tau)$ be the closed polyhedral cone obtained from $C(\pi,\tau)$ by relaxing the strict inequalities. 

\begin{definition}\label{defn:NBF2}
	We call the collection of cones $\{\sigma(\pi,\tau): (\pi, \tau) \in \fS_{d+1}\times \fS_{d}\}$, together with all of their faces, the \emph{nested Braid fan}, denoted by $\Br_d^2.$
\end{definition}

\begin{example}
	Let $(\pi,\tau) = (3241, 231).$ Then $(\pi^{-1},\tau^{-1})=(4213,312).$ Thus, the $C(\pi,\tau)$ is the collection of $\xx \in W_d$ satisfying
\begin{enumerate}
	\item $x_4 < x_2 < x_1 < x_3$, and
	\item $x_3-x_1<x_2-x_4<x_1-x_2.$
\end{enumerate}
\end{example}

We will use similar ideas to those presented in the last section to prove that $\Br_d^2$ is a complete projective fan by showing it is the normal fan of a family of polytopes, which will be constructed below. We start by choosing two strictly increasing sequences
\[ \balpha = (\alpha_1, \alpha_2, \dots, \alpha_{d+1}) \in \RR^{d+1} \quad \text{and} \quad \bbeta = (\beta_1, \beta_2, \dots, \beta_{d}) \in \RR^{d}.\]
We then pick $M, N > 0.$ The basic idea of the construction is to take the $M$-th dilation of the usual permutohedron $\Perm(\balpha)$ and then replace each of its vertices with an $N$-th dilation of $\Perm(\bbeta)$ under a suitable coordinate system. This will give us $d! (d+1)!$ vertices. Below is the precise construction. For any $(\pi, \tau) \in \fS_{d+1} \times \fS_d$, we define
\begin{equation}\label{equ:defnv}
v_{\pi,\tau}^{(\balpha,\bbeta),(M,N)} := M \sum_{i=1}^{d+1} \alpha_i \ee_{\pi^{-1}(i)} + N \sum_{i=1}^d \beta_i \ff_{\tau^{-1}(i)}^\pi.
\end{equation}
(Note that $\sum_{i=1}^{d+1} \alpha_i \ee_{\pi^{-1}(i)} = v_\pi^\balpha$ is a vertex of $\Perm(\balpha).$) 
We omit $(\balpha,\bbeta)$ from the superscript, and only write $v_{\pi,\tau}^{(M,N)}$ if $(\balpha,\bbeta) = \left( (1,2,\dots,d+1), (1,2,\dots, d) \right).$ 

After rearranging coordinates, we get the following expression:
\begin{equation}
v_{\pi,\tau}^{(\balpha,\bbeta),(M,N)} = \sum_{i=1}^{d+1} (M \alpha_i + N(\beta_{\tau(i-1)} - \beta_{\tau(i)})) \ \ee_{\pi^{-1}(i)},
	\label{equ:expansion} \end{equation}
where by convention we let $\beta_{\tau(0)} = \beta_{\tau(d+1)} =0.$ We would like to have the coefficients of $\ee_{\pi^{-1}(i)}$ in the above expansion increase strictly as $i$ increases, for any $(\pi,\tau)$. 
	If this happens, we say $(M,N) \in \RR_{>0}^2$ is an \emph{appropriate choice} for $(\balpha, \bbeta)$. 
It is not hard to see that for fixed $(\balpha, \bbeta)$, any pair $(M,N)$ satisfying $M >> N$ is an appropriate choice. 

\begin{definition}\label{defn:usual2}
	Suppose $(\balpha, \bbeta) \in \RR^{d+1} \times \RR^d$ is a pair of strictly increasing sequences $(\balpha, \bbeta) \in \RR^{d+1} \times \RR^d$ and $(M,N) \in \RR_{>0}^2$ is an \emph{appropriate choice} for $(\balpha, \bbeta)$. We define the \emph{usual nested permutohedron}
	\begin{equation}
\Perm(\balpha, \bbeta; M,N) := \conv\left(v_{\pi,\tau}^{(\balpha,\bbeta),(M,N)}:\quad (\pi,\tau) \in \fS_{d+1}\times \fS_d\right).
\label{equ:defnusual2}
	\end{equation}
	In particular, if $\balpha = (1, 2,\dots, d+1)$ and $\bbeta=(1,2,\dots, d),$ we call the polytope a \emph{regular nested permutohedron}, denoted by $\Pi_d^2(M,N).$ (So $v_{\pi,\tau}^{(M,N)}$ are vertices for $\Pi_d^2(M,N).$)

\end{definition}
We remark that similar to the definition of $\Perm(\balpha)$, the above definition does not directly say that each $v_{\pi,\tau}^{(\balpha,\bbeta),(M,N)}$ is a vertex of $\Perm(\balpha, \bbeta; M,N).$ However, it will be shown to be true in Proposition \ref{prop:fanofusual2} below.

One sees that $\Perm(\balpha,\bbeta;M,N)$ lies in the hyperplane $\sum_{i=1}^{d+1} x_i = M\sum_{i=1}^{d+1} \alpha_i$, which is a translation of $V_d,$ and $\Perm(\balpha, \bbeta; M, N)$ is centralized if and only if $\sum_{i=1}^{d+1} \alpha_i=0,$ which is the situation we will focus on.

\begin{example}\label{ex:nestedperm}
One can show that $(M,N)=(4,1)$ is an appropriate choice for $(\balpha=(1,2,3,4), \bbeta=(1,2,3))$. Thus, $\Pi_3^2(4,1)$ is a nested regular permutohedron. 
See Figure \ref{fig:2polytopes} for a picture of it together with a picture of the regular permutohedron $\Pi_3$ as a comparison. 

Let $(\pi,\tau)=(3241,231)$. Then $(\pi^{-1}, \tau^{-1})=(4213, 312)$. Thus, the vertex of $\Pi_3^2(4,1)$ associated to $(3241,231)$ is 
\[
v_{3241,231}^{(4,1)}=
4(1\ee_4+2\ee_2+3\ee_1+4\ee_3) + 1(1(\ee_3-\ee_1)+2(\ee_2-\ee_4)+3(\ee_1-\ee_2)) = (14,7,17,2).
\]
We can compute all vertices of $\Pi_3^2(4,1)$ this way, and they are
\[
(3,7,11,19),(2,9,10,19),(1,10,11,18),(1,9,13,17),(2,7,14,17),(3,6,13,18),
\]
and all of their permutations.
\end{example}


\begin{proposition}\label{prop:fanofusual2}
	Suppose $(\balpha, \bbeta) \in \RR^{d+1} \times \RR^d$ is a pair of strictly increasing sequences and $(M,N) \in \RR_{>0}^2$ is an appropriate choice for $(\balpha, \bbeta)$. 

	Then for each $(\pi,\tau) \in \fS_{d+1} \times \fS_d,$ the point $v_{\pi,\tau}^{(\balpha,\bbeta),(M,N)}$ is a vertex of $\Perm(\balpha,\bbeta),$ and the normal cone of $\Perm(\balpha,\bbeta;M,N)$ at $v_{\pi,\tau}^{(\balpha,\bbeta);(M,N)}$ is $\sigma(\pi,\tau).$ Therefore, the nested Braid fan $\Br_d^2$ is the normal fan of the nested usual permutohedron $\Perm(\balpha, \bbeta; M, N).$ Hence, $\Br_d^2$ is a complete projective fan in $W_d$.
\end{proposition}

\begin{proof}
	Similar to the proof of Proposition \ref{prop:fanofusual}, it is enough to show that for any $\ww \in C(\pi, \tau)$ (assuming $(\pi,\tau)$ is fixed), 
\begin{equation}
\label{equ:strictineq}
	\left\langle \ww, v_{\pi,\tau}^{(\balpha,\bbeta),(M,N)} \right\rangle > \left\langle \ww, v_{(\pi',\tau')}^{(\balpha,\bbeta),(M,N)} \right\rangle, \quad \forall (\pi,\tau) \neq (\pi',\tau') \in \fS_{d+1} \times \fS_d.
\end{equation}	
We will prove the above inequality by introducing an intermediate product and showing
\begin{equation}\label{equ:strictineq1}
	\left\langle \ww, v_{\pi,\tau}^{(\balpha,\bbeta),(M,N)} \right\rangle > \left\langle \ww, v_{\pi,\tau'}^{(\balpha,\bbeta),(M,N)} \right\rangle  > \left\langle \ww, v_{(\pi',\tau')}^{(\balpha,\bbeta),(M,N)} \right\rangle. 
\end{equation}

Similar as before, 
we let $u_i = w_{\pi^{-1}(i)}$ for each $i$ and express $\ww$ as in \eqref{equ:wexp0}.
Then $\ww \in C(\pi, \tau)$ means that
\begin{enumerate}
\item $u_1 < u_2 < \dots < u_{d+1}$, and
\item $u_{\tau^{-1}(1)+1} - u_{\tau^{-1}(1)} < u_{\tau^{-1}(2)+1}-u_{\tau^{-1}(2)} < \cdots < u_{\tau^{-1}(d)+1} - u_{\tau^{-1}(d)}.$
\end{enumerate}
Expression \eqref{equ:wexp0}, together with \eqref{equ:expansion}, allows us to compute products in \eqref{equ:strictineq} easily. 
Then the second inequality in \eqref{equ:strictineq1} follows from the Rearrangement Inequality (Lemma \ref{lem:rearr}), condition (1) above and the fact that $(M,N)$ is an appropriate choice. 
Next, we see
the first inequality in \eqref{equ:strictineq1} holds if and only if
\[\sum_{i=1}^{d+1} u_i \left(\beta_{\tau(i-1)} - \beta_{\tau(i)}\right) > \sum_{i=1}^{d+1} u_i \left(\beta_{\tau'(i-1)} - \beta_{\tau'(i)}\right).\]
After rearranging summations, the above inequality becomes
\[ \sum_{j=1}^d \beta_j \left(u_{\tau^{-1}(j)+1} - u_{\tau^{-1}(j)}\right) >\sum_{j=1}^d \beta_j \left(u_{(\tau')^{-1}(j)+1} - u_{(\tau')^{-1}(j)}\right),\]
which follows from the Rearrangement Inequality, condition (2) above and the fact that $\bbeta$ is strictly increasing.
\end{proof}

Proposition \ref{prop:fanofusual2} provides one natural way to define generalized nested permutohedra.

\begin{definition}\label{defn:gen2} A polytope in $V_d$ (or in an affine plane that is a translation of $V_d$) is a \emph{generalized nested permutohedron} if its normal fan is a coarsening of $\Br_d^2.$

\end{definition}

\subsection*{Combinatorics of $\Br_d^2$}
Our next goal is to determine the combinatorics of the fan $\Br^2_d$ which is equivalent to those of the face lattice of $\Pi_d^2(M,N)$. 
The following poset arises naturally in our discussion.

\begin{definition}\label{defn:osp}
An \emph{ordered (set) partition} of $[d+1]$ is an ordered tuple of disjoint subsets whose union is $[d+1]$, i.e. $\cT=(S_1,\cdots,S_k)$ with $S_i\subset [d+1]$ for all $1\leq i\leq k$ and $S_1\sqcup\cdots \sqcup S_k = [d+1]$. 

The \emph{ordered (set) partition poset}, denoted by $\OO_{d+1}$, is the poset on all ordered set partitions of $[d+1]$ ordered by refinement.
This is a ranked poset of rank $d.$ It has a maximum, the trivial partition, $\hat{1}=([d+1])$, but doesn't have a minimum. 

It has $(d+1)!$ minimal elements, one for each permutation $\pi \in \fS_{d+1}$ considered as an ordered set partition of singletons:
\[ \cT(\pi) := ( \{\pi^{-1}(1)\}, \{\pi^{-1}(2)\}, \dots, \{\pi^{-1}(d+1)\}).\]

We denote by $\overline{\OO_{d+1}}$ the poset obtained from $\OO_{d+1}$ by removing the maximum element.
\end{definition}
\begin{remark}
We are going to write ordered set partitions by using numbers separated by bars. For instance, the ordered partition $\cT=(\{3,4\},\{1,5\},\{2,6,7\})$ will be written as $34|15|267$. It is important to keep in mind that the numbers between bars form a set, hence their order is irrelevant.  We have $\cT(3721456) = 4|3|1|5|6|7|2 \le 34|15|267$.
\end{remark}

Recall we define $\ee_S$ for each $S \in \BB_{d+1}.$ For each element $\cT = (S_1,\cdots,S_k) \in \OO_{d+1}$, we define
\begin{equation}
	\ee_{\cT} := \sum_{i} i \ee_{S_i}.
	\label{equ:defneT}
\end{equation}
For instance if $\cT=34|15|267$, then 
$\ee_{\cT} =  1 \cdot \ee_{34} + 2 \cdot \ee_{15} + 3 \cdot \ee_{267} = (2,3,1,1,2,3,3).$

We have the following result that is analogous to Proposition \ref{prop:charbr}.
\begin{proposition}\label{prop:charbr2}
The rays, i.e., $1$-dimensional cones, of the Braid fan $\Br_d^2$ are given by $\ee_\cT$ for all $\cT \in \overline{\OO_{d+1}}$.  
Furthermore, a $k$-set of rays $\{\ee_{\cT_1},\cdots, \ee_{\cT_k}\}$ spans a $k$-dimensional cone in $\Br_d^2$ if and only if the sets $\cT_1,\dots, \cT_k$ form a $k$-chain in $\overline{\OO_{d+1}}$.

In particular, the maximal cones in $\Br_d^2$ are in bijection with the maximal chains in $\overline{\OO_{d+1}}.$ Hence, $\Br_d^2$ is simplicial. 
\end{proposition}

As each maximal cone $\sigma(\pi,\tau)$ of $\Br_d^2$ is indexed by $(\pi,\tau) \in \fS_{d+1} \times \fS_d$, and maximal chains in $\overline{\OO_{d+1}}$ are obtained from maximal chains in $\OO_{d+1}$ by removing the top element,
we will prove the above proposition by providing a bijection between $(\pi,\tau) \in \fS_{d+1}\times \fS_d$ and maximal chains in $\OO_{d+1}.$

We first observe that the rank-$0$ element $\cT(\pi): \pi^{-1}(1)|\pi^{-1}(2)|\cdots|\pi^{-1}(d+1)$ contains $d$ bars, and any element of rank $r$ in the interval $[\cT(\pi), \hat{1}]$ can be obtained from $\cT(\pi)$ by removing an $r$-subset of the $d$ bars. Conversely, any element of rank $r$ arises this way. This gives the following lemma. 

\begin{lemma}\label{lem:localbool}
	For each $\pi \in \fS_{d+1},$ the interval $[\cT(\pi), \hat{1}]$ is isomorphic to the Boolean algebra $\BB_{d}.$ Hence, the poset $\OO_{d+1}$ is \emph{locally Boolean}, i.e., all of its intervals are Boolean algebras. 
\end{lemma}
	
Moreover, the discussion above provides a natural way to construct a desired bijection for the proof of Proposition \ref{prop:charbr2}.
\begin{notation}
We represent each $(\pi,\tau)$ with the following diagram, denoted by $\cD(\pi,\tau)$:
\[
\pi^{-1}(1)\stackrel{\tau(1)}{|}  \pi^{-1}(2) \stackrel{\tau(2)}{|}  \cdots  \stackrel{\tau(d-1)}{|}  \pi^{-1}(d) \stackrel{\tau(d)}{|}\pi^{-1}(d+1).
\]
\end{notation}

\begin{definition}
	Let $(\pi,\tau) \in \fS_{d+1} \times \fS_d,$ we define $\ch(\pi,\tau)$ to be the unique maximal chain in $[\cT(\pi), \hat{1}]$ that is obtained in the following way: 
	\begin{enumerate}
		\item Let $\cD(\pi,\tau; 0) = \cD(\pi,\tau)$. 
		\item For each $1 \le r \le d,$ we let $\cD(\pi,\tau;r)$ be the diagram obtained from $\cD(\pi,\tau;r-1)$ by removing the bar labelled by $r.$
		\item For each $0 \le r \le d,$ ignoring the labels on bars gives an ordered set partition in $[\cT(\pi),\hat{1}]$ of rank $r$, and we denote it by $\cT(\pi,\tau; r).$
		\item Let $\ch(\pi,\tau)$ be the maximal chain formed by $\{ \cT(\pi,\tau;r) \ : \ 0 \le r \le d\}.$
	\end{enumerate}

\end{definition}

%

\begin{example}\label{ex:nestedperm2}
	Let $(\pi,\tau) = (3241, 231).$ Then $\cD(\pi,\tau)$ is the diagram \begin{equation}\label{eq:diagram}
4\stackrel{2}{|}2\stackrel{3}{|}1\stackrel{1}{|}3,
\end{equation}
	and $\ch(\pi,\tau)$ is as shown in the box on the left side of Figure \ref{fig:maxchain}, where the arrows demonstrate the procedure we describe above. In the middle of the figure (or the third column of the figure), we list the rays $\ee_{\cT(\pi,\tau;i)}$ associated with $\cT(\pi,\tau;r)$ for each $r.$
	Finally, in the fourth column, we show the difference between any two consecutive associated rays, which turns out to be important.
\begin{figure}[h]
	\begin{tikzpicture}[scale=0.9]

\node at (-2.4,3.8) {${\cD(\pi,\tau; r)}$};
\node at (-2.4,0) {$4\stackrel{2}{|}2\stackrel{3}{|}1\stackrel{1}{|}3$};	
\node at (-2.4,0.5) {$\uparrow$};
\node at (-2.4,1) {$4\stackrel{2}{|}2\stackrel{3}{|}1 \ \ 3$};
\node at (-2.4,3/2) {$\uparrow$};
\node at (-2.4,4/2) {$4 \ \ 2\stackrel{3}{|}1 \ \ 3$};
\node at (-2.4,5/2) {$\uparrow$};
\node at (-2.4,6/2) {$4 \ \ 2 \ \ 1 \ \ 3$};

\node at (-1,0) {$\longrightarrow$};
\node at (-1,1) {$\longrightarrow$};
\node at (-1,2) {$\longrightarrow$};
\node at (-1,3) {$\longrightarrow$};

\node at (0,3.8) {${\cT(\pi,\tau; r)}$};
\node at (0,0) {$4|2|1|3$};
\node at (0,0.5) {$|$};
\node at (0,1) {$4|2|13$};
\node at (0,3/2) {$|$};
\node at (0,4/2) {$42|13$};
\node at (0,5/2) {$|$};
\node at (0,6/2) {$4213$};

\draw (-0.6,-0.3) -- (0.6,-0.3) -- (0.6,3.3) -- (-0.6,3.3) -- cycle;

\begin{scope}[xshift=0.5cm]
\node at (3,0) {$\ee_4+2\ee_2+3\ee_1+4\ee_3$};
\node at (3,1) {$\ee_4+2\ee_2+3\ee_1+3\ee_3$};
\node at (3,2) {$\ee_4+\ee_2+2\ee_1+2\ee_3$};
\node at (3,3) {$\ee_4+\ee_2+\ee_1+\ee_3$};
\node at (3,3.8) {$\ee_{\cT(\pi,\tau; r)}$};
\end{scope}

\begin{scope}[xshift=1cm]
\node at (6.6,3.8) {$\ee_{\cT(\pi,\tau; r-1)} - \ee_{\cT(\pi,\tau;r)}$};
\node at (6.6, 0.5) {$\ee_{\{3\}} = \ee_3$};
\node at (6.6, 1.5) {$\ee_{\{213\}} = \ee_2 + \ee_1 + \ee_3$};
\node at (6.6, 2.5) {$\ee_{\{13\}} = \ee_1 + \ee_3$};

\node at (10,3.8) {$\Gamma_{\tau^{-1}(r)}^\pi$};
\node at (10, 0.5) {$\Gamma_{3}^\pi =\{3\}$};
\node at (10, 1.5) {$\Gamma_{1}^\pi = \{213\}$};
\node at (10, 2.5) {$\Gamma_{2}^\pi =\{13\}$};

\node at (8.7, 0.5) {$\longleftarrow$};
\node at (8.7, 1.5) {$\longleftarrow$};
\node at (8.7, 2.5) {$\longleftarrow$};
\end{scope}
\end{tikzpicture}
\caption{Maximal chain and associated rays}
\label{fig:maxchain}
\end{figure}
\end{example}

One may notice that in Figure \ref{fig:maxchain} that the differences $\ee_{\cT(\pi,\tau; r-1)} - \ee_{\cT(\pi,\tau;r)}$ can be understood in a more systematic way. We use the following notation.
\begin{notation}
	Fix $\pi \in \fS_{d+1}$. For $1 \le i \le d,$ let 
	\[ \Gamma_i^{\pi} := \pi^{-1}(i, d+1] = \{ \pi^{-1}(j) \ : \ i < j \le d+1\}.\]
\end{notation}

The following lemma is clear from the construction of $\ch(\pi,\tau)$. So we omit its proof.
\begin{lemma}\label{lem:diff}
	The map $(\pi,\tau) \to \ch(\pi,\tau)$ is a bijection from $\fS_{d+1}\times \fS_d$ to maximal chains of $\OO_{d+1}$ (or equivalently, to maximal chains of $\overline{\OO_{d+1}}$).

	Furthermore, for each $1 \le r \le d,$ 
\begin{equation}\label{equ:diff}
\ee_{\cT(\pi,\tau; r-1)} - \ee_{\cT(\pi,\tau;r)} = \ee_{\Gamma_{\tau^{-1}(r)}^\pi}.
\end{equation}
\end{lemma}


\begin{example}
In our running example, where $(\pi,\tau) = (3241, 231),$ we have
\[ \Gamma_1^\pi = \{ 213\}, \quad \Gamma_2^\pi = \{ 13\}, \quad \Gamma_3^\pi = \{3\}.\]
Then the differences $\ee_{\cT(\pi,\tau; r-1)} - \ee_{\cT(\pi,\tau;r)}$ can be computed using \eqref{equ:diff} as shown in the last column of Figure \ref{fig:maxchain}.
For instance, $\tau^{-1}(3)=2$ implies that $\ee_{\Gamma_2^\pi} = \ee_{\{13\}}$ gives the difference vector $\ee_{\cT(\pi,\tau; 2)} - \ee_{\cT(\pi,\tau;3)}$.
\end{example}

\begin{proof}[Proof of Proposition \ref{prop:charbr2}] It is enough to show that $\sigma(\pi,\tau)$ is spanned by the rays $\ee_{\cT(\pi,\tau; r)}$, where $0 \le r \le d-1,$ associated to non-maximum elements in the maximal chain $\ch(\pi,\tau).$ (All the conclusions in the proposition follow from it.)

It follows from the definition that $\sigma(\pi,\tau)$ is the collection of $\xx \in W_d$ satisfying
\begin{equation}\label{eq:chainineq2}
	0\leq (\Delta \xx)^\pi_{\tau^{-1}(1)}\leq (\Delta \xx)^\pi_{\tau^{-1}(2)}\leq\cdots \leq (\Delta \xx)^\pi_{\tau^{-1}(d)}.
\end{equation}
The rays of $\sigma(\pi,\tau)$ are obtained by having one strict inequality in \eqref{eq:chainineq2} and equalities in the rest, i.e., by
\begin{equation}\label{equ:chainineqr}
	0 = (\Delta \xx)^\pi_{\tau^{-1}(1)}=(\Delta \xx)^\pi_{\tau^{-1}(2)}=\cdots=(\Delta \xx)^\pi_{\tau^{-1}(r)}<(\Delta \xx)^\pi_{\tau^{-1}(r+1)} =\cdots= (\Delta \xx)^\pi_{\tau^{-1}(d)} =1.
\end{equation}
The right hand side can be any positive constant as the solution will just be off by a scale; hence, we let it be $1.$ As there is a unique solution (if one exists) to \eqref{equ:chainineqr}, it is enough to verify $\ee_{\cT(\pi,\tau;r)}$ is a solution. Indeed, by the construction of $\cT(\pi,\tau;r),$ the followings are true: 
\begin{enumerate}
	\item If $i \le r,$ then $\pi^{-1}(\tau^{-1}(i)+1)$ is in the same block of $\cT(\pi,\tau;r)$ as $\pi^{-1}(\tau^{-1}(i))$. 
	\item If $i > r,$ then $\pi^{-1}(\tau^{-1}(i)+1)$ is in the block of $\cT(\pi,\tau;r)$ that follows the block where $\pi^{-1}(\tau^{-1}(i))$ is in.  
\end{enumerate}
Then the desired conclusion follows from the definition of $\ee_\cT$ (see \eqref{equ:defneT}) for any ordered set partition $\cT$. 
\end{proof}

As a summary, we have associated three objects to each pair of $(\pi,\tau)$. The proofs of Propositions \ref{prop:fanofusual2} and \ref{prop:charbr2} tell us the connection between them, which are summarized in the diagram below. 
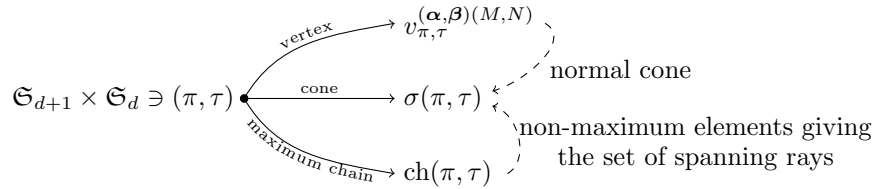
\begin{figure}[h]
\begin{tikzpicture}
\node[left] at (-1,2) {$\fS_{d+1}\times\fS_d\ni(\pi,\tau)$};
\draw[->, postaction={decorate,decoration={raise=0.5ex,text along path,text align=center, text={|\tiny|vertex}}}] (-1,2) to [out=60, in=190]  (1,3);
\draw[fill](-1,2) circle [radius=0.05];
\node[right] at (1,3) {$v_{\pi,\tau}^{(\balpha,\bbeta)(M,N)}$};
\draw[->,dashed] (3,3) to [out=330, in=15]  (2.3,2.1);
\draw[->,dashed] (2.5,1) to [out=30, in=345] (2.3,1.9);
\node[right] at (1,2) {$\sigma(\pi,\tau)$};
\node at (4,2.4) {normal cone};

\node[right] at (1,1) {$\ch(\pi,\tau)$};
\node at (5,1.6) {non-maximum elements giving}; 
\node at (5,1.2) {the set of spanning rays};
\draw[->, postaction={decorate,decoration={raise=0.5ex,text along path,text align=center, text={|\tiny|cone}}}] (-1,2) -- (1,2);
\draw[->, postaction={decorate,decoration={raise=-1ex,text along path,text align=center, text={|\tiny|maximum chain}}}] (-1,2) to [out=300, in=170]  (1,1);
\end{tikzpicture}
\caption{Relation between objects associated to $(\pi,\tau)$}
\label{fig:relation}
\end{figure}

\begin{example} \label{ex:nestedperm3}
	Let $(\pi,\tau)=(3241,231).$ As shown in Examples \ref{ex:nestedperm} and \ref{ex:nestedperm2} that $v_{3241,231}^{(4,1)} = (14,7,17,2)$, the maximal chain $\ch(3241,231)$ and its associated rays are given in Figure \ref{fig:maxchain}. So the normal cone of $\Pi_3^2(4,1)$ at the vertex $v_{3241,231}^{(4,1)}$ is $\sigma(3241,231)$. It is spanned by the rays associated to non-maximum elements in $\ch(3241,231),$ which are the three vectors on the bottom of the middle column in Figure \ref{fig:maxchain}. This helps us to find three facet-defining inequalities for $\Pi_3^2(4,1):$
\begin{align*}
	\left\langle \ee_{42|13}, \xx \right\rangle = 2x_1+x_2+2x_3+x_4 &\leq	\left\langle \ee_{42|13}, (14,7,17,2) \right\rangle =  71\\
	\left\langle \ee_{4|2|13}, \xx \right\rangle = 3x_1+2x_2+3x_3+x_4 &\leq\left\langle \ee_{4|2|13}, (14,7,17,2) \right\rangle = 109\\
	\left\langle \ee_{4|2|1|3}, \xx \right\rangle = 3x_1+2x_2+4x_3+x_4 &\leq\left\langle \ee_{4|2|1|3}, (14,7,17,2) \right\rangle = 126.
\end{align*}
\end{example}
\subsection*{Inequality description of usual nested permutohedra}

It follows from Propositions \ref{prop:fanofusual2} and \ref{prop:charbr2} and Definition \ref{defn:gen2} that any generalized nested permutohedron in $\RR^{d+1}$ is defined by the linear system in the form of
	\begin{equation} \label{equ:linear2}
		\langle \ee_{[d+1]}, \xx \rangle = \langle \1, \xx \rangle \ = \ b_{[d+1]}, \quad \text{and} \quad 
		\langle \ee_\cT, \xx \rangle \ \le \ b_\cT, \quad \forall \cT \in \overline{\OO_{d+1}}. 
	\end{equation}
	Note that $[d+1]$ can be considered as the maximal element in $\OO_{d+1}$ which is an ordered set partition, and also can be considered as the maximal element in $\BB_{d+1}$, which is a set. Either way, $\ee_{[d+1]}$ represents the all-one vector $\1.$ Hence, we may consider all the $b$'s appearing in \eqref{equ:linear2} as a vector $\bb \in \RR^{\OO_{d+1}}$ where indices are elements in $\OO_{d+1}.$ 
	
	It is interesting to obtain results that are analogous to those in Theorems \ref{thm:centralsub} and \ref{thm:submodrestate}. We will do this in the next part. Before that we focus on usual nested permutohedra. When we introduced usual nested permutohedra, we defined them as convex hull of all of their vertices in \eqref{equ:defnusual2}. Now we can give an inequality description in the form of \eqref{equ:linear2} by giving an explicit description for $\bb$. It turns out each coordinate $b_\cT$ is determined by its structure type.

	\begin{definition}\label{defn:type}
		Let $\cT = S_1| S_2| \dots | S_{k+1} \in \OO_{d+1}.$ We define the \emph{structure type} of $\cT$, denoted by $\Type(\cT)$, to be the sequence $(t_0=0, t_1, t_2, \dots, t_{k+1}=d+1),$ where 
		\[ t_i = \sum_{j=1}^i |S_j|, \quad \text{for $1 \le i \le k$}.\]
(We can also understand $t_i (1 \le i \le k)$ as the position number of the $i$th bar in $\cT.$)		
\end{definition}

\begin{theorem}\label{thm:facetdes}
		Suppose $(\balpha, \bbeta) \in \RR^{d+1} \times \RR^d$ is a pair of strictly increasing sequences $(\balpha, \bbeta) \in \RR^{d+1} \times \RR^d$ and $(M,N) \in \RR_{>0}^2$ is an appropriate choice for $(\balpha, \bbeta)$. Suppose $\bb \in \RR^{\OO_{d+1}}$ is defined as follows: for each $\cT \in \OO_{d+1},$ if $\Type(\cT)=(t_0, t_1,t_2,\dots, t_k, t_{k+1}),$ let
\begin{equation}
	b_{\cT} = M \left( \sum_{i=1}^{k+1} i \sum_{j=t_{i-1}+1}^{t_i} \alpha_j\right) + N \sum_{j=d-k+1}^d \beta_j.
	\label{equ:usualb}
\end{equation}
Then the linear system \eqref{equ:linear2} defines the usual nested permutohedron $\Perm(\balpha, \bbeta; M,N)$. 
\end{theorem}

\begin{proof} 
	We have discussed after Definition \ref{defn:usual2} that $\Perm(\balpha, \bbeta;M,N)$ lies on the hyperplane $\sum_{i=1}^{d+1} x_i = M \sum_{i=1}^{d+1} \alpha_i.$ This, together with Propositions \ref{prop:fanofusual2} and \ref{prop:charbr2}, implies that it is enough to verify that $\bb$ defined by \eqref{equ:usualb} satisfies 
	\begin{enumerate}
		\item $b_{[d+1]} = M \sum_{i=1}^{d+1} \alpha_i;$ and
		\item if $\cT \in \overline{\OO_{d+1}}$ is in the maximal chain $\ch(\pi,\tau),$ then $b_\cT = \left\langle \ee_\cT, v_{\pi,\tau}^{(\balpha,\bbeta),(M,N)} \right\rangle.$
	\end{enumerate}
Suppose $\cT = [d+1]$ is the maximal element of $\OO_{d+1}.$ Then $k=0$ and its structure type is $(0, d+1).$ The right hand side of \eqref{equ:usualb} becomes $M \sum_{j=1}^{d+1} \alpha_j$ as desired.

Suppose $\cT = S_1 | S_2 | \cdots | S_{k+1} \in \overline{\OO_{d+1}}$ has structure type $(t_0, t_1, \dots, t_{k+1}),$ and it belongs to the maximal chain $\ch(\pi,\tau).$ (Note that the choice of $(\pi,\tau)$ is not unique.) 
%
It is easy to see that $\cT$ is of rank $d-k.$ Hence, it is the rank-$(d-k)$ element $\cT(\pi,\tau;d-k)$ of the maximal chain $\ch(\pi,\tau).$ It follows from the construction of $\ch(\pi,\tau)$ that $\cD(\pi,\tau; d-k)$ is the following diagram:
\[
\pi^{-1}(1) \pi^{-1}(2) \cdots \pi^{-1}(t_1) \stackrel{\tau(t_1)}{|}  \pi^{-1}(t_1+1) \cdots \pi^{-1}(t_2) \stackrel{\tau(t_2)}{|}  \cdots  \cdots \stackrel{\tau(t_k)}{|}  \pi^{-1}(t_k+1) \cdots \pi^{-1}(d+1).
\]
Since $\cT = \cT(\pi,\tau;d-k) = S_1 |S_2 |\cdots|S_{k+1}$ is obtained from $\cD(\pi,\tau;d-k)$ by removing labels on the bars, and the labels on the bars have to be the largest $k$ elements in $[d],$ the followings are true
\begin{align}
	&	\{ \tau(t_1), \tau(t_2),\dots, \tau(t_k)\} = \{d-k+1, d-k+2, \dots, d\}; \label{equ:barcond} \\
&	S_i = \{ \pi^{-1}(j) : t_{i-1}+1 \le j \le t_i\} \quad \forall 1 \le i \le k+1. \label{equ:Si}
\end{align}
It follows from \eqref{equ:Si} that $\ee_\cT = \sum_{i=1}^{k+1} i \sum_{j=t_{i-1}+1}^{t_i} \ee_{\pi^{-1}(j)}$. Using this and \eqref{equ:defnv}, we obtain
\begin{align*}
	\left\langle \ee_\cT, v_{\pi,\tau}^{(\balpha,\bbeta),(M,N)} \right\rangle =&  M \left( \sum_{i=1}^{k+1} i \sum_{j=t_{i-1}+1}^{t_i} \alpha_j\right) + N \left\langle \ee_\cT, \sum_{i=1}^d \beta_{\tau(i)} \ff_{i}^\pi \right\rangle 
\end{align*}
However,
$\langle \ee_\cT, \ff_{i}^\pi \rangle = \langle \ee_\cT, \ee_{\pi^{-1}(i+1)} - \ee_{\pi^{-1}(i)} \rangle$ is $0$ if $\pi^{-1}(i)$ and $\pi^{-1}(i+1)$ are in the same block of $\cT$, and is $1$ otherwise, in which case they are in two consecutive blocks. One checks that the latter situation happens if and only if
$i = t_j$ for some $1 \le j \le k$, which is a position where a bar is placed. Therefore, it follows from \eqref{equ:barcond} that 
$\displaystyle \left\langle \ee_\cT, \sum_{i=1}^d \beta_{\tau(i)} \ff_{i}^\pi \right\rangle = \sum_{j=d-k+1}^d \beta_j$ as desired.
\end{proof}

\begin{example} 
	We apply Theorem \ref{thm:facetdes} to the regular nested permutohedron $\Pi_3^2(4,1)$ first studied in Example \ref{ex:nestedperm}. Clearly, the polytope lies in 
	\[ \langle \ee_{[4]}, \xx \rangle = x_1 + x_2 + x_3 + x_4 =b_{[d+1]} = M \sum_{i=1}^4 \alpha_i = 4 (1+2+3+4) = 40.\]
We then compute some of the inequalities:
\begin{align*}
\langle \ee_{23|4|1},\xx\rangle = 3x_1+x_2+x_3+2x_4\leq 4\Big(1(1+2)+2(3)+3(4)\Big)+1\Big(2+3\Big)&= 90\\
\langle \ee_{14|23},\xx\rangle = x_1+2x_2+2x_3+x_4\leq 4\Big(1(1+2)+2(3+4)\Big)+1\Big(3\Big)&= 71\\
\langle \ee_{4|123},\xx\rangle = 2x_1+2x_2+2x_3+x_4\leq 4\Big(1(1)+2(2+3+4)\Big)+1\Big(2+3\Big)&= 81.
\end{align*}

\end{example}

\subsection*{Deformation cone of $\Br_d^2$}

Finally, we are going to present results that are analogous to results on $\Def(\Br_d)$ and the Submodularity Theorem for generalized permutohedron discussed in Section \ref{sec:GP}. We first apply Proposition \ref{prop:reduxfan} to determine $\Def(\Br_d^2)$, or equivalently, the deformation cone of a \emph{centralized} usual nested permutohedron. 

In order to apply Proposition \ref{prop:reduxfan}, one needs to describe pairs of adjacent maximal cones in $\Br_d^2.$ By Proposition \ref{prop:charbr2}, this is equivalent to describing pairs of maximal chains in $\OO_{d+1}$ that only differ at a non-maximum element. 
Suppose $\ch_1$ and $\ch_2$ forms such a pair of maximal chains in $\OO_{d+1}$. Let $\cT = \ch_1 \setminus \ch_2$ and $\cT'= \ch_2\setminus \ch_1.$ Then $\cT$ and $\cT'$ are of same rank, say $r$, where $0 \le r < d.$ In this case, we say $\{ \ch_1, \ch_2\}$ is a pair of \emph{($r$-)adjacent} maximal chains in $\OO_{d+1}$. 

\begin{figure}[t]
\begin{tikzpicture}

	\begin{scope}[scale=0.8]
    \draw [black,fill] (0,0) circle [radius = 0.08];
    \draw[black, fill] (0,-1) circle [radius=0.08];
    \draw[black, fill] (1,1) circle [radius=0.08];
    \draw[black, fill] (-1,1) circle [radius=0.08];
    \draw[black, fill] (0,2) circle [radius=0.08];
    \draw[black, fill] (0,3) circle [radius=0.08];
    \draw (0,-1) -- (0,0) -- (1,1) -- (0,2);
    \draw (0,0) -- (-1,1) -- (0,2)--(0,3);
    \node[above] at (0,3.2) {\vdots};
    \node[below] at (0,-1) {\vdots};
    \node[below right] at (0,-1) {$\cT_{r-2}$};
    \node[below right] at (0,0) {$\cT_{r-1}$};
    \node[right] at (1,1) {$\cT'_{r}=\cT'$};
    \node[left] at (-1,1) {$\cT_{r}=\cT$};
    \node[right] at (0,2) {$\cT_{r+1}$};
    \node[right] at (0,3) {$\cT_{r+2}$};

    \node[scale=1.5] at (-1.7,3) {$\ch_1$};
    \node[scale=1.5] at (2.1,3) {$\ch_2$};
    \node at (0,-2.5) {a diamond with $0 < r < d$};
\end{scope}

	\begin{scope}[xshift=6cm, yshift=-2cm, scale=0.8]
    \draw[black, fill] (1,1) circle [radius=0.08];
    \draw[black, fill] (-1,1) circle [radius=0.08];
    \draw[black, fill] (0,2) circle [radius=0.08];
    \draw[black, fill] (0,3) circle [radius=0.08];
    \draw[black, fill] (0,4) circle [radius=0.08];
    \draw[black, fill] (0,5) circle [radius=0.08];
    \draw (1,1) -- (0,2) -- (0,3)--(0,4)--(0,5);
    \draw (-1,1) -- (0,2);
    \node[above] at (0,5.2) {\vdots};
    \node[right] at (1,1) {$\cT'_{r}$};
    \node[left] at (-1,1) {$\cT_{r}$};
    \node[right] at (0,2) {$\cT_{r+1}$};
    \node[right] at (0,3) {$\cT_{r+2}$};
    \node[right] at (0,4) {$\cT_{r+1}$};
    \node[right] at (0,5) {$\cT_{r+2}$};
    
    \node[scale=1.5] at (-1.7,3) {$\ch_1$};
    \node[scale=1.5] at (2.1,3) {$\ch_2$};
    \node at (0,0) {a ``\ren'' shape with $r=0$};
\end{scope}

\end{tikzpicture}
\caption{Two possibilities}
\label{fig:2poss}
\end{figure}
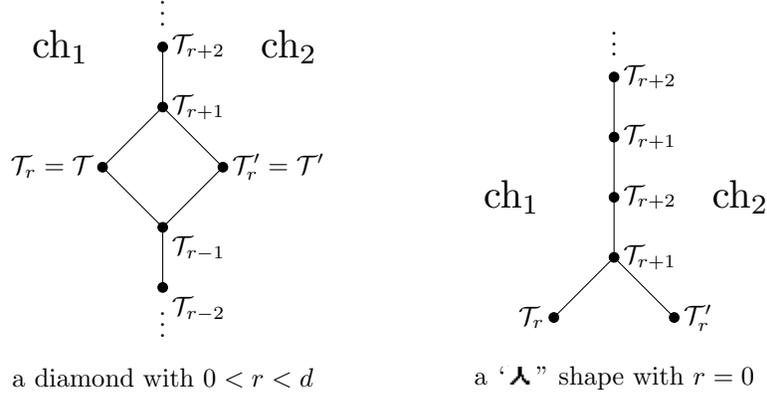

\begin{lemma}\label{lem:adjch}
	Suppose $\{ \ch_1, \ch_2\}$ is a pair of \emph{$r$-adjacent} maximal chains in $\OO_{d+1}$. Then (by Lemma \ref{lem:diff}) there exists a unique pair $(\pi_i, \tau_i) \in \fS_{d+1} \times \fS_d$ such that $\ch_i = \ch(\pi_i,\tau_i)$ for each $i.$ There are two situations:
	\begin{enumerate}
		\item (The diamond situation:) If $0 < r < d,$ then $\ch_1$ and $\ch_2$ form a diamond shape as shown on the left of Figure \ref{fig:2poss}. Furthermore, we have
		\begin{equation}\label{equ:dia}
		\pi_1 = \pi_2, \quad \text{and} \quad (r,r+1) \circ \tau_1 = \tau_2,
	\end{equation}
			where $(r,r+1)$ is the transposition that exchanges $r$ and $r+1.$
		\item (The \ren situation:) If $r=0,$ then $\ch_1$ and $\ch_2$ form a \ren (reads ``ren'') shape as shown on the right of Figure \ref{fig:2poss}. Suppose $\cT_{r+1} = \cT_1$, the minimum common element of $\ch_1$ and $\ch_2$, has its only $2$-element-block in $i$th position, that is, 
			\begin{equation}
			\cT_1 =s_1|s_2|\cdots|s_{i-1}|s_i s_i'|s_{i+1}|\cdots|s_d. 	
				\label{equ:rank1}
			\end{equation}
			Then  
		\begin{equation}\label{equ:ren}
		\tau_1 = \tau_2, \quad \tau_1(i)=1=\tau_2(i), \quad \text{and} \quad (i,i+1) \circ \pi_1 = \pi_2.
	\end{equation}
	\end{enumerate}
	Moreover, if $(\pi_1,\tau_1)$ and $(\pi_2,\tau_2)$ satisfy either \eqref{equ:dia} or \eqref{equ:ren}, then $\ch(\pi_1,\tau_1)$ and $\ch(\pi_2,\tau_2)$ are adjacent maximal chains in $\OO_{d+1}.$
\end{lemma}

\begin{proof}
Suppose $r \neq 0.$ Then $\ch_1$ and $\ch_2$ has the same minimum element, say $\cT(\pi).$ Thus, $\ch_1$ and $\ch_2$ are two maximal chains in the maximal interval $[\cT(\pi), \hat{1}],$ and form a diamond. Hence, $\pi_1 = \pi = \pi_2.$
By the construction of $\ch(\pi,\tau)$, we must have
\begin{equation}
 \tau_1^{-1}(r) = \tau_2^{-1}(r+1), \text{ and } \tau_1^{-1}(r+1) = \tau_2^{-1}(r), 
	\label{equ:tautran}
\end{equation}
and $\tau_1^{-1}(i) = \tau_2^{-1}(i)$ for $i \neq r, r+1.$ This is equivalent to $(r, r+1) \circ \tau_1 = \tau_2.$

 Suppose $r=0$, and $\ch_1$ and $\ch_2$ are as shown on the right of Figure \ref{fig:2poss}. Assume further that $\cT_1$ is given by \eqref{equ:rank1}. 
 Then the two minimal elements in $\ch_1$ and $\ch_2$ are
 \[ s_1|s_2|\cdots|s_{i-1}|s_i| s_i'|s_{i+1}|\cdots|s_d, \quad \text{and} \quad s_1|s_2|\cdots|s_{i-1}|s_i'| s_i|s_{i+1}|\cdots|s_d.\]
 Then \eqref{equ:ren} follows from the construction of $\ch(\pi,\tau).$

 Finally, the last assertion can be easily verified.
 \end{proof}

 Using the connection between adjacent chains of $\OO_{d+1}$ and adjacent vertices of the regular nested permutohedron $\Perm(\balpha,\bbeta;M,N),$ we immediately have the following result.
 \begin{corollary}
	 The two vertices $v_{\pi_1,\tau_1}^{(\balpha,\bbeta),(M,N)}$ and $v_{\pi_2,\tau_2}^{(\balpha,\bbeta),(M,N)}$ are adjacent, i.e., form an edge, if and only if either \eqref{equ:dia} or \eqref{equ:ren} holds.
 \end{corollary}

 Each pair of adjacent maximal chains described in Lemma \ref{lem:adjch}, via its correspondence with a pair of adjacent maximal cones in $\Br_d^2,$ is associated with an inequality as described in Definition \ref{defn:fanineq}.  
In the lemma below, we describe this association explicitly. 

\begin{lemma}\label{lem:adjchineq} Assume all the hypothesis in Lemma \ref{lem:adjch}. Let $\sigma_i = \sigma(\pi_i,\tau_i)$ be the maximal cone in $\Br_d^2$ that is in bijection with $\ch_i$ for $i=1,2.$ 
	 \begin{enumerate}
		 \item (The diamond situation:) Suppose $\ch_1$ and $\ch_2$ form a diamond shape as shown on the left of Figure \ref{fig:2poss}. Then the associated inequality $I_{\{\sigma_1, \sigma_2\}}(\bb)$ is
			 \begin{equation}
		 b_{\cT_{r+1}} + b_{\cT_{r-1}} \le b_{\cT_r} + b_{\cT'_r}.
				 \label{equ:diaineq}
			 \end{equation}
We call such an inequality a \emph{diamond submodular inequality}.

\item (The \ren situation:)  Suppose $\ch_1$ and $\ch_2$ form a \ren shape as shown on the right of Figure \ref{fig:2poss}, and $\cT_1$ is given by \eqref{equ:rank1}. (We know that $\tau_1=\tau_2.$) Let $\tau=\tau_1=\tau_2,$ and then let $p = \tau(i-1)$ and $q = \tau(i+1).$ (By convention, let $\tau(0)=0$ and $\tau(d+1)=d+1.$) 
	Then the associated inequality $I_{\{\sigma_1, \sigma_2\}}(\bb)$ is
	\begin{equation}
	2b_{\cT_1} + \underbrace{\left(b_{\cT_{p-1}}-b_{\cT_{p}}\right)}_{\text{replaced with $b_{[d+1]}$ if $p=0$}} + \underbrace{\left(b_{\cT_{q-1}}-b_{\cT_{q}}\right)}_{\text{eliminated if $q=d+1$}} \ \le b_{\cT_0} + b_{\cT_0'}.  \label{equ:renineq}
	\end{equation}
	We call such an inequality a \emph{\ren inequality}.
	 \end{enumerate}
	 In both situations, we assume $b_{[d+1]} = b_{\hat{1}} = 0.$
 \end{lemma}

 \begin{remark}\label{rem:balance}
	 There are (at least) two ways to see why $b_{[d+1]} =0:$ First, we are discussing the centralized permutohedra which all lie in $V_d$, where each point has its coordinates sum to $0.$ Second, $\ee_{[d+1]} = 0$ in $W_d$ is not a spanning ray of any maximal cone in $\Br_d^2.$

	 The reason we keep $b_{[d+1]}$ in our expression \eqref{equ:renineq} (as well as its later reformulations) is to keep the expression \emph{balanced}, which means if we replace each $b_\cT$ with $\ee_\cT$ in \eqref{equ:renineq}, then the left side gives the same vector as the right side considering both are vectors in $\RR^{d+1}$ (instead of $W_d$). (It will be clear from the proof of Lemma \ref{lem:adjchineq} below why \eqref{equ:renineq} is balanced.) 
 \end{remark}

 \begin{proof}[Proof of Lemma \ref{lem:adjchineq}]
	 For the diamond situation, as it was shown in the proof of Lemma \ref{lem:adjch} that $\tau_1$ and $\tau_2$ satisfy \eqref{equ:tautran}.  Then it follows from the second part of Lemma \ref{lem:diff} that 
 \[\ee_{\cT_{r-1}} - \ee_{\cT_r} = \ee_{\cT(\pi,\tau_1; r-1)} - \ee_{\cT(\pi,\tau_1;r)} =\ee_{\cT(\pi,\tau_2; r)} - \ee_{\cT(\pi,\tau_2;r+1)} = \ee_{\cT'_r} - \ee_{\cT_{r+1}},\] 
which implies that
\[ \ee_{\cT_{r+1}} + \ee_{\cT_{r-1}} = \ee_{\cT_r} + \ee_{\cT'_r}.\]
This gives us the diamond submodular inequality \eqref{equ:diaineq}.

We now consider the \ren  situation. Without loss of generality, we may assume
\begin{equation}
 \cT_0 = s_1|s_2|\cdots|s_{i-1}|s_i| s_i'|s_{i+1}|\cdots|s_d, \quad \text{and} \quad \cT_0' = s_1|s_2|\cdots|s_{i-1}|s_i'| s_i|s_{i+1}|\cdots|s_d.  \label{equ:rank0}
 \end{equation}
 Then
 \begin{equation}
 \ee_{\cT_0} -\ee_{\cT_1} = \ee_{s_i'} + \sum_{j=i+1}^d \ee_{s_j} \quad \text{and} \quad 
 \ee_{\cT_0'} -\ee_{\cT_1} = \ee_{s_i} + \sum_{j=i+1}^d \ee_{s_j}.
	 \label{equ:reneq}
 \end{equation}
 Note that by the second part of Lemma \ref{lem:diff}, we have
 \begin{align*}
	 \ee_{s_i}+\ee_{s_i'} + \sum_{j=i+1}^d \ee_{s_j} =& \begin{cases}
		 \ee_{\cT_{p-1}} - \ee_{\cT_p} \quad & \text{if $p \neq 0$} \\ 
		 \ee_{[d+1]} =\ee_{\hat{1}} \quad & \text{if $p=0$};\end{cases}  \\
	 \text{and} \quad \sum_{j=i+1}^d \ee_{s_j} =& \begin{cases}
		 \ee_{\cT_{q-1}} - \ee_{\cT_q} \quad & \text{if $q \neq d+1$} \\ 
		 0 \quad & \text{if $q=d+1$}.\end{cases}
 \end{align*}
 One sees that the sum of the left hand sides of the above two equalities equals to the sum of the right hand sides of the two equalities in \eqref{equ:reneq}. This gives us an equality involving $\ee_{\cT_0}, \ee_{\cT_0'}, \ee_{\cT_1},$ $\ee_{\cT_{p-1}} - \ee_{\cT_p}$ and $\ee_{\cT_{q-1}} - \ee_{\cT_q}.$ Rearranging terms and applying Definition \ref{defn:fanineq} yields the desired inequality \eqref{equ:renineq}.
 \end{proof}

 We combine results in part (1) of Lemmas \ref{lem:adjch} and \ref{lem:adjchineq} to obtain the following reformulated description for diamond submodular inequalities. The proof is straightforward, so is omitted. 
 \begin{corollary}\label{cor:diamond}
	 Let $(\pi,\tau) \in \fS_{d+1}\times \fS_d$ and $1 \le r < d.$ If we let $\tau' = (r, r+1)\circ \tau,$ then $\ch(\pi,\tau)$ and $\ch(\pi, \tau')$ form a pair of $r$-adjacent maximal chains and their associated diamond submodular inequality can be written as:
	 \begin{equation}
 b_{\cT(\pi,\tau; r-1)} - b_{\cT(\pi,\tau; r)} \le b_{\cT(\pi,\tau'; r)} - b_{\cT(\pi,\tau'; r+1)}.  
		 \label{equ:refdia}
	 \end{equation}
 \end{corollary}

 While the diamond submodular inequalities are in a simple form that is easy to describe, the inequalities arising from the \ren situation are relative messy. Fortunately, given the diamond submodular inequalities, in particular their reformulations given in Corollary \ref{cor:diamond}, we only need to consider a subset of the ones from the \ren situation. In fact, for each rank-$1$ $\cT_1$ element of $\OO_{d+1},$ we only need one inequality constructed from a \ren shape containing $\cT_1.$  

 \begin{lemma} \label{lem:ess}
	Let $\cT = \cT_1 = s_1|s_2|\cdots|s_{i-1}|s_i s_i'|s_{i+1}|\cdots|s_d$ be an element $\OO_{d+1}$ of rank $1.$ It covers exactly two rank-$0$ elements $\cT_0$ and $\cT_0'$ as given in \eqref{equ:rank0}.
We also let
\[ \cS := s_1 s_2 \cdots s_{i-1} | s_i s_i' | s_{i+1}\cdots s_d \]
be the maximum element that is above $\cT$ and still contains $\{s_i,s_i'\}$ as a single block. Note that $\cS$ is of rank $d-1$ if $i=1$ or $d,$ and is of rank $d-2$ otherwise.

Assume all the hypothesis in Lemma \ref{lem:adjchineq}, including those assumption in part (ii) for the \ren situation. 
Moreover, assume further the common rank-$1$ element of $\ch_1$ and $\ch_2$ is the $\cT_1$ given above. Clearly, $\cT_0, \cT_0'$ are the two rank-$0$ elements. 
\begin{enumerate}[label=(\alph*)]
	\item 
If $S$ is a common element of $\ch_1$ and $\ch_2,$ then the associated \ren inequality $I_{\{\sigma_1, \sigma_2\}}(\bb)$ becomes: 
\begin{equation}\label{equ:essineq}
	2 b_\cT + b_\cS \le b_{\cT_0} + b_{\cT_0'} + \underbrace{b_{[d+1]}}_{\text{eliminated if $i=1$}}.
\end{equation}
\item If we do not assume $\cS$ is a common element of $\ch_1$ and $\ch_2,$ then the associated \ren inequality $I_{\{\sigma_1, \sigma_2\}}(\bb)$ can be deduced from the inequality \eqref{equ:essineq} and all the diamond submodular inequalities.
\end{enumerate}
\end{lemma}
We remark that the term $b_{[d+1]}$ in \eqref{equ:essineq} can be removed even when $i \neq 1$ as we have the assumption $b_{[d+1]}=0$. However, as we stated in Remark \ref{rem:balance}, we keep this term to make our inequality balanced.

\begin{proof}
	Both parts of the lemma can be verified in three cases: (i) When $i=1$, and $\cS$ is of rank $d-1;$ (ii) when $i=d,$ and $\cS$ is of rank $d-1;$ (iii) when $i\neq 1$ or $d$, and $\cS$ is of rank $d-2.$ The proofs for all cases are similar. Therefore, we only present the one for case (i) where $i=1.$

	\begin{enumerate}[label=(\alph*)]
		\item 
	Since $i=1,$ we immediately have that $p = \tau(i-1)=0.$ The element $\cS = s_1 s_1'|s_2 s_3 \cdots s_d$ is of rank $d-1,$ and so is the second element $\cT_{d-1}$ from the top on the maximal chain $\ch_1=\ch(\pi_1,\tau_1)=\ch(\pi_1,\tau).$ By the construction of $\ch(\pi_1,\tau),$ the top element $\cT_d = \hat{1}=[d+1]$ in $\ch_1$ was obtained from $\cS$ by removing the bar at the $2$nd position, which means $q = \tau(i+1)=\tau(2)=d$. Plugging this information into \eqref{equ:renineq}, we obtain \eqref{equ:essineq}.
	
\item We may assume $d \ge 2$ since if $d=1,$ there is only one inequality arising from the \ren situation (and no diamond submodular inequalities). 
	As in the previous part, we have $p=\tau^{-1}(i-1)=0.$ Since $d \ge 2,$ we have $q = \tau(i+1) = \tau(2) \neq d+1.$ Hence, the associated inequality is
	\[ 2b_{\cT_1} + b_{[d+1]} + \left(b_{\cT_{q-1}}-b_{\cT_{q}}\right) \le b_{\cT_0} + b_{\cT_0'}. \]
	Comparing it with \eqref{equ:essineq}, we see that it is enough to show the following inequality can be deduced from all the diamond submodular inequalities:
	\begin{equation}\label{equ:mdia} b_{\cT_{q-1}}-b_{\cT_{q}} \le b_\cS - b_{[d+1]}.
	\end{equation}
	Clearly, if $q=d,$ the equality holds. So we assume $q < d.$ Let $\gamma_0 = \tau,$ and for each $1 \le i \le d-q,$ let $\gamma_i = (q+i, q+i-1) \circ \gamma_{i-1}.$ 
Then $\gamma_i(2) = q+i.$ In particular, $\gamma_{d-q}(2) = d.$
Therefore, we have	
\[ \cT_{q-1} = \cT(\pi_1,\gamma_0; q-1), \cT_{q} = \cT(\pi_1,\gamma_0; q), \quad \text{and} \quad \cS = \cT(\pi_1, \gamma_{d-q}; d-1), [d+1] = \cT(\pi_1,\gamma_{d-q}; d).\]
	
	We apply Corollary \ref{cor:diamond} ($d-q$) times with $\pi = \pi_1$ and $(\tau,\tau') = (\gamma_{i-1}, \gamma_i)$ for all $1 \le i \le d-q,$ and obtain $d-q$ diamond submodular inequalities in the form of \eqref{equ:refdia}.
	Adding these inequalities together yields the desired inequality \eqref{equ:mdia}.
\end{enumerate}
\end{proof}

We see that the inequality \eqref{equ:essineq} only depends on $\cT = \cT_1$, as $\cT_0, \cT_0'$ and $\cS$ are all determined by $\cT.$ Hence, we denote the inequality \eqref{equ:essineq} by $I_{\cT}(\bb)$, and call it the \emph{essential inequality} associated with the rank-$1$ element $\cT=\cT_1.$

\begin{example}\label{ex:ess}
	In Figure \ref{fig:renex}, we give an example of the proof of part (b) of Lemma \ref{lem:ess}, showing how to use diamond submodular inequalities to reduce an arbitrary \ren inequality to an essential one. 
We start with the pair of adjacent chains given by the circled elements, $\ch_1= (\cT_0,\cT_1,\cT_2,\cT_3,\cT_4)$ and $\ch_2=(\cT_0',\cT_1,\cT_2,\cT_3,\cT_4)$.  Their associated \ren inequality is 
\begin{equation}\label{equ:exineq}
	2b_{\cT_1}+ b_{[5]} + (b_{\cT_1} - b_{\cT_2})  \leq b_{\cT_0}+b_{\cT_0'}.
\end{equation}
However, using the two diamonds shown in Figure \ref{fig:renex}, we obtain that 
\[ b_{\cT_1} - b_{\cT_2}\le b_{\cT_2'} - b_{\cT_3}\leq b_{\cT_3'}-b_{\cT_4} = b_{53|142} - b_{[5]}.\]
This shows that the \ren inequality \eqref{equ:exineq} can be deduced from the above two diamond submodular inequalities and the following inequality:
\[ 	2b_{\cT_1}+ b_{[5]} + (b_{53|142} - b_{[5]}) = 2b_{\cT_1} + b_{53|142}  \leq b_{\cT_0}+b_{\cT_0'},\]
which is exactly the essential \ren inequality associated to $\cT_1.$
\begin{figure}[h]
\centering
\includegraphics[scale=0.8]{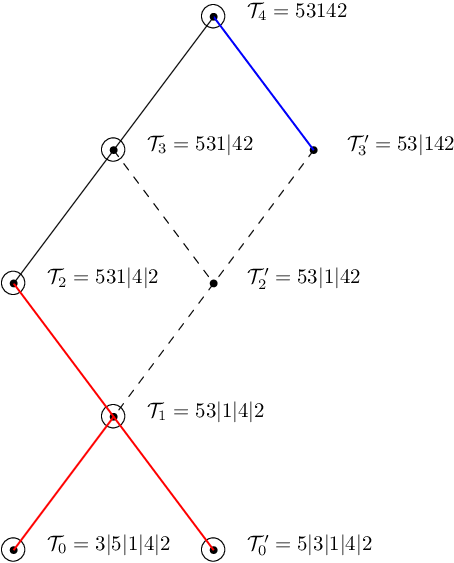}
\caption{Illustration of Lemma \ref{lem:ess}}
\label{fig:renex}
\end{figure}
\end{example}

All the discussion above, together with Proposition \ref{prop:reduxfan} and Remark \ref{rem:equivsub}, gives us the first main result of this part. 
Recall that $\OO_{d+1}$ is locally boolean (Lemma \ref{lem:localbool}), so we can define $\wedge$ and $\vee$ on any pair of elements in an interval. 
\begin{theorem}\label{thm:centralsub2}
The deformation cone of the nested Braid fan (or centralized nested regular permutohedron) is the collection of $\bb \in \RR^{\overline{\OO_{d+1}}}$ satisfying the following conditions: 
\begin{enumerate}
	\item (Local submodularity) All the diamond submodular inequalities on $\OO_{d+1}$ are satisfied, or equivalently, for any maximal interval $[\cT(\pi), \hat{1}]$,
\begin{equation}\label{equ:bsub2}
	b_{\cS \vee \cT} + b_{\cS \wedge \cT} \le b_{\cS} + b_{\cT}, \quad \forall \cS, \cT \in [\cT(\pi),\hat{1}].
\end{equation}
\item (  \ren condition) For any rank-$1$ element $\cT \in \overline{\OO_{d+1}},$  its associated essential inequality $I_{\cT}(\bb)$ holds. 
\end{enumerate}
For both parts, we assume $b_{\hat{1}} = b_{[d+1]} = 0.$
\end{theorem}

If we remove the condition $b_{[d+1]}=0$ which corresponds to the centralized cases, we get a theorem that characterize all generalized nested permutohedra, analogous to Theorem \ref{thm:submodrestate}.
\begin{theorem}\label{thm:sub2}
	For $\bb \in \RR^{\OO_{d+1}}$ satisfying the local submodularity condition and the \ren condition described in Theorem \ref{thm:centralsub2}, the linear system:
	\begin{equation} \label{equ:linear3}
		\langle \ee_{[d+1]}, \xx \rangle = \langle \1, \xx \rangle \ = \ b_{[d+1]}, \quad \text{and} \quad 
		\langle \ee_\cT, \xx \rangle \ \le \ b_\cT, \quad \forall \cT \in \overline{\OO_{d+1}} 
	\end{equation}
	defines a generalized nested permutohedron in $\RR^{d+1},$ and any generalized nested permutohedron arises this way uniquely.

		Furthermore, if a polytope $P \in \RR^{d+1}$ is defined by a tight representation \eqref{equ:linear3}, then $P$ is a generalized nested permutohedron if and only if $\bb \in \RR^{\OO_{d+1}}$ satisfies the local submodularity condition and the \ren condition. 
\end{theorem}

\begin{proof}
	The proof is similar to that of Theorem \ref{thm:submodrestate}, and we only give a sketch of the proof for the first part. 
	The one-to-one correspondence between centralized nested permutohedra and $\bb$'s satisfying the local submodularity condition and the \ren condition with $b_{[d+1]}=0$ is established by Theorem \ref{thm:centralsub2}. 

	Suppose $\bb \in \RR^{\OO_{d+1}}$. Let $k = \frac{b_{[d+1]}}{d+1}$ and define a new vector/function $\bb' \in \RR^{\OO_{d+1}}$ by
\[ \bb'_\cT = \bb_\cT - k \cdot \card(\cT), \quad \forall \cT \in \OO_{d+1},\]
where $\card(\cT) = \left\langle \ee_{\cT}, \1 \right\rangle = \sum_{i} i |S_i|,$ if $\cT = (S_1,S_2 ,\cdots ,S_k)$.

	Let $P$ and $Q$ be the polytopes defined by the linear system \eqref{equ:linear3} with vectors $\bb$ and $\bb'$ respectively. Then we have the following facts: 
\begin{enumerate}
	\item $\bb'_{[d+1]} =0.$
	\item $\bb'$ satisfies the local submodularity condtion and the \ren condition if and only if $\bb$ satisfies these two conditions as well. 
	\item $Q =\tilde{P} = P - k \1$ is the centralized version of $P.$
\end{enumerate}
Facts (1) and (3) are straightforward to check, and fact (2) follows from that all the inequalities we describe are balanced. (See Remark \ref{rem:balance}.) We see the first conclusion of the theorem follows from these facts and the arguments in the first paragraph.
\end{proof}

\begin{example}\label{ex:unexampleII}

	Recall that the polytope $P \subset \RR^4$ considered in Example \ref{ex:unexample}. We already mentioned that $P$ is the cube whose vertices are $(1,1,1,3)$ and $(0,2,2,2)$ and their permutations. Furthermore, by discussing edge directions, we conclude that $P$ is not a generalized permutohedron. 
	Another way to see this is by looking at the normal cones of $P$ at each vertex. For example, let $\sigma$ be the normal cone of $P$ at the vertex $(0,2,2,2).$ One can show that $\sigma$ is spanned by $3$ rays in $\Br_3$: $\ee_{\{2,3\}}, \ee_{\{3,4\}}, \ee_{\{2,4\}}.$ However, another ray $\ee_{ \{2,3,4\}}$ of $\Br_3$ is in the middle of $\sigma.$ As a result, $\sigma$ cuts through $6$ maximal cones of $\Br_3$, and thus is not a union of maximal cones in $\Br_3.$ Figure \ref{fig:nonexample}/(A) depicts a slice of these $6$ cones where the shaped region corresponds to the normal cone $\sigma.$ Hence, the normal fan of $P$ does not refine $\Br_3,$ and by Proposition \ref{prop:coarser}, $P$ is not a generalized permutohedron.

\begin{figure}[h]
\centering
\begin{subfigure}{.5\textwidth}
  \centering
  \includegraphics[width=.6\linewidth]{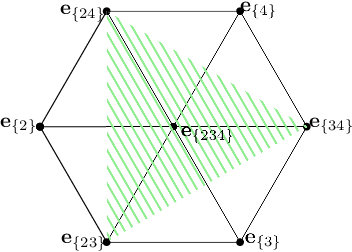}
  \caption{$\sigma$ in the Braid fan $\Br_3$}
  \label{fig:sub1}
\end{subfigure}%
\begin{subfigure}{.5\textwidth}
  \centering
  \includegraphics[width=.6\linewidth]{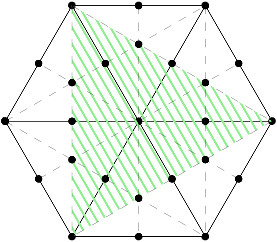}
  \caption{$\sigma$ in the nested Braid fan $\Br_3^2$}
  \label{fig:sub2}
\end{subfigure}
\caption{Comparison of a normal cone $\sigma$ in $\Br_3$ and $\Br_3^2$}
\label{fig:nonexample}
\end{figure}
However, in $\Br_3^2,$ each maximal cone of $\Br_3$ was subdivided into $6$ cones. Figure \ref{fig:nonexample}/(B) shows how the maximal cones in Figure \ref{fig:nonexample}/(A) were subdivided, where the dark dots are rays in $\Br_3^2.$ One sees that $\sigma$ is a union of maximal cones in $\Br_3^2.$ Similarly, all the other normal cones of $P$ are unions of maximal cones in $\Br_3^2.$ Hence the normal fan of $P$ refines $\Br_3^2$. Thus, $P$ is a generalized \emph{nested} permutohedron. 
\end{example}

\section{Chiseling Constructions} \label{sec:chisel}
Victor Reiner asked whether it is true that $\Br_d^2$ is the barycentric subdivision of $\Br_d.$ The main purpose of this section is to give an affirmative answer to his question. 
We start by introducing the concept of \emph{chiseling} off faces of a polytope, which will be used to construct the \emph{barycentric subdivision}. 
%
\begin{definition}\label{defn:chisel}
Suppose $G$ is a face of a $d$-dimensional polytope $P\subset V.$ Let $F_1,\cdots F_l$ be the facets containing $G$ with primitive outer normals $\bfa_1,\cdots, \bfa_l$ respectively; in other words, $\bfa_1,\cdots, \bfa_l$ are the spanning rays of the normal cone $\ncone(G,P)$. Define the \emph{chiseling direction of $P$ at $G$} to be 
\[ \bfa_G :=\bfa_1+\cdots+ \bfa_l.\] Furthermore, let $b_G$ be the scalar such that
$G = P \cap \{ \xx \in V \ : \ \langle \bfa_G, \xx \rangle = b_G \}$; equivalently, 
\[ b_G := \max_{\xx \in P} \langle \bfa_G, \xx \rangle.\]
For any sufficiently small $\epsilon>0$ such that $\{\xx: \langle \bfa_G,\xx\rangle < b_G-\epsilon\}$ contains all vertices of $P$ not in $G$, we define $P_\epsilon:=P\cap \{\xx \ : \ \langle \bfa_G,\xx\rangle \leq b_G-\epsilon\}$ to be the polytope obtained by \emph{chiseling $G$ off $P$ (at distance $\epsilon$)}.  
We call the facet $P \cap \{\xx: \langle \bfa_G,\xx\rangle = b_G-\epsilon\}$ of $P_\epsilon$ created by this process the \emph{facet obtained by chiseling $G$ off $P$}.

Let $G_1, \dots, G_k$ be faces of $P.$ We say $G_1, \dots, G_k$ can be \emph{simultaneously chiseled off $P$ at distance $\epsilon$} if for any $1 \le i < j \le k$ the facet obtained by chiseling $G_i$ off $P$ at distance $\epsilon$ has no intersection with the facet obtained by chiseling $G_j$ off $P$ at distance $\epsilon.$ 
\end{definition}

\begin{remark}
We remark that the term ``chiseling'' follows from \cite[Section 6]{bruns}. Other terms such as ``shaving'' and ``truncating'' with the same meaning are also used in the literature.
\end{remark}

See Figure \ref{fig:chisel} for a picture of chiseling off a vertex from a polygon. 
\begin{figure}[h]
\centering
\includegraphics{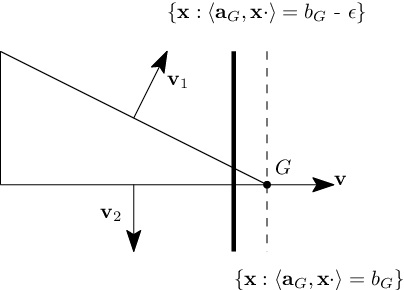}
\caption{Chiseling off a vertex from a polygon. The set $\{\xx: \langle \bfa_G,\xx\rangle = b_G-\epsilon\}$ is given by the thick line, whereas the other is the dashed line.}
\label{fig:chisel}
\end{figure}
It is easy to see that $G_1, \dots, G_k$ can be simultaneously chiseled off $P$ at a sufficient small distant $\epsilon>0$ if and only if $G_1,\dots,G_k$ are pairwise disjoint. 


\begin{remark}
To chisel a face $G$ of a polytope $P$ correspond to make a \emph{stellar subdivision} on $\Sigma(P)$ along $\ncone(G,P)$ (See \cite[Section III.2]{ewald}). 
\end{remark}

Suppose $\Sigma$ is a projective fan in $W$ such that $0 \in \Sigma.$
The following algorithm gives one way to obtain the barycentric subdivision of $\Sigma$. 
\begin{algorithm}\label{alg:bary}
	\begin{enumerate}
		\item[(0)] 
	Let $P_0=P$ be a $d$-polytope whose normal fan is $\Sigma$. 
	
\item 	Let $\epsilon_1 > 0$ be a sufficiently small number such that we can simultaneously chisel all vertices off $P_0$ at distance $\epsilon_1$, and let $P_1$ be the polytope obtained from $P_0$ by applying these chiselings.

\item 	Let $E_1, \dots, E_m$ be the edges of $P_1$ that come from $P_0$, that is, edges that are not created from the chiselings done in the steps above. 
	Let $\epsilon_2>0$ 
	be a sufficiently small number such that we can simultaneously chisel $E_1, \dots, E_m$ off $P_1$ at distance $\epsilon_2$, and let $P_2$ be the polytope obtained from $P_1$ by applying these chiselings.

\item  \dots

\item[\vdots] 

\item[(d)] Let $F_1, \dots, F_d$ be the $(d-1)$-dimensional faces of $P_{d-1}$ that come from $P_0.$ 
	Let $\epsilon_d >0$ 
	be a sufficiently small number such that we can simultaneously chisel $F_1, \dots, F_n$ off $P_{d-1}$ at distance $\epsilon_d$, and let $P_d$ be the polytope obtained from $P_{d-1}$ by applying these chiselings.
	\end{enumerate}
\end{algorithm}
It follows from \cite[Definition 2.5, Section III.2]{ewald} that the normal fan of $P_d$ obtained by Algorithm \ref{alg:bary} is the \emph{barycentric subdivision} of $\Sigma.$

\begin{remark}
For any $k > 0, $ the set of all $k$-dimensional faces of a polytope $P$ cannot be simultaneously chiseled, since they are not pairwise disjoint. However, they do become pairwise disjoint after $k$ steps of Algorithm \ref{alg:bary}. For instance, the set of all edges of $P_0$ is clearly not pairwise disjoint, but the resulting edges after chiseling all vertices of $P_0$ are pairwise disjoint and can be simultaneously chiseled.
\end{remark}

\begin{remark}
The barycentric subdivision of the normal fan of a polytope should not be confused with the barycentric subdivision of the polytope itself. For any polytope $P$, its barycentric subdivision is a triangulation of $P$, whereas the barycentric subdivision of a fan is again a fan. Furthermore Algorithm \ref{alg:bary} shows that it preserves projectivity.
 \end{remark}
The following lemma, which follows immediately from the construction of $P_d,$ will be useful in our discussion.
\begin{lemma}\label{lem:Pd}
  The resulting polytope $P_d$ (of Algorithm \ref{alg:bary}) is a full-dimensional polytope in the same $d$-dimensional affine space as $P_0,$ and is defined by the following linear system:
\[ \langle \bfa_G, \xx \rangle \le b_G - \epsilon_{\dim(G)+1}, \quad \text{for all nonempty proper faces $G$ of $P$},\]
where $\bfa_G$ and $b_G$ are as defined in Definition \ref{defn:chisel}, and $\epsilon_i$'s are the chiseling distances given in Algorithm \ref{alg:bary}.
\end{lemma}

We are now ready to state the main result, Theorem \ref{thm:bary2}, of this section, preceded by a classical result, Theorem \ref{thm:bary1}, that is related to Reiner's question,
 Recall that the \emph{standard $d$-simplex} is $\Delta_d := \conv\{ \ee_1, \dots, \ee_{d+1}\}$. Theorem \ref{thm:bary1} follows from Remark 6.6 in \cite{PosReiWil}.
\begin{theorem} \label{thm:bary1}
	The Braid fan $\Br_d$ is the barycentric subdivision of the normal fan $\Sigma(\Delta_d)$ of $\Delta_d$.
\end{theorem}


\begin{theorem}\label{thm:bary2}
	The nested Braid fan $\Br_d^2$ is the barycentric subdivision of the Braid fan $\Br_d,$ and thus is the second barycentric subdivision of $\Sigma(\Delta_d).$
\end{theorem}


As a warmup, we give a proof for Theorem \ref{thm:bary1} in which we use well-known facts about faces of the standard simplices.
\begin{proof}[Proof of Theorem \ref{thm:bary1}]
We apply Algorithm \ref{alg:bary} to $P_0 = \Delta_d$, making sure that the chiseling distances $\epsilon_i$'s satisfy:
\begin{equation}
	\epsilon_1 < \frac{1}{2}, \quad \text{and for $2 \le i \le d,$} \quad \epsilon_i < \frac{\epsilon_{i-1}}{2}.
\label{equ:tiny}
\end{equation}
It is sufficient to show that the resulting polytope $P_d$ is a usual permutohedron. In order to do this, we will apply Lemma \ref{lem:Pd} to find an inequality description for $P_d.$

	First, note that $\Delta_d$ is a full-dimensional polytope in the affine space
	\[ 	\langle \ee_{[d+1]}, \xx \rangle = \langle \1, \xx \rangle \ = \ 1.\]
Next, faces of $\Delta_d$ are naturally indexed by subsets of $[d+1]$. For each $S \subseteq [d+1]$, the corresponding face is $\conv\{\ee_i: i\in S\}$, which we denote by $G_S.$
Recall that for any $\emptyset \neq S \subsetneq [d+1],$ the normal cone of $\Delta_d$ at $G_S$ is generated by $\{-\ee_j:j\notin I\}$. Hence, the chiseling direction of $\Delta_d$ at $G_S$ is $\sum_{j\notin I} -\ee_j$. As the normal fan is defined in $W_d$, where $\ee_{[d+1]} = \sum_{i\in[d+1]} \ee_i=0$, this chiseling direction can be written as $\ee_S=\sum_{i\in S} \ee_i$. Finally, let \[ b_S := \max_{\xx \in \Delta_d} \langle \ee_S, \xx \rangle = 1.\] 

Therefore, by Lemma \ref{lem:Pd}, the polytope $P_d$ is given by the following linear system:
\begin{equation} \label{eq:chiselex}
		\langle \ee_{[d+1]}, \xx \rangle = \langle \1, \xx \rangle \ = \ 1, \quad \text{and} \quad 
		\langle \ee_S, \xx \rangle \ \le \ b_S-\epsilon_{\dim(G_S)+1}=1-\epsilon_{|S|}, \quad \forall \emptyset \neq S \subsetneq [d+1].
	\end{equation}
	It follows from \eqref{equ:tiny} that 
\begin{equation}
 \epsilon_d < \epsilon_{d-1}-\epsilon_d < \epsilon_{d-2}-\epsilon_{d-1} < \cdots < \epsilon_1- \epsilon_2 < 1 - \epsilon_1,
	\label{equ:tiny1}
\end{equation}
using which one can show that the right hand side of \eqref{eq:chiselex} is a submodular function, so that Theorem \ref{thm:submodular} applies. More directly, one can show that $P_d$ is the usual permutohedron $\Perm(\balpha)$ with
\[ \balpha = \Big(\epsilon_d, \epsilon_{d-1}-\epsilon_d, \epsilon_{d-2}-\epsilon_{d-1}, \dots, \epsilon_1- \epsilon_2, 1 - \epsilon_1\Big).\]
\end{proof}

We are going to prove Theorem \ref{thm:bary2} in a parallel fashion. Before that, we give the following preliminary lemma.
\begin{lemma}\label{lem:bij}
	There exists a one-to-one correspondence between $(d-k)$-faces of $\Pi_d$ and ordered set partitions with $k+1$ parts in $\overline{\OO_{d+1}}$ such that if we let $G_\cT$ be the face corresponds to the ordered set partition $\cT,$ then the chiseling direction of $\Pi_d$ at $G_\cT$ is $\ee_\cT.$ 
\end{lemma}

\begin{proof}
	It follows from Propositions \ref{prop:fanofusual} and \ref{prop:charbr} that each $(d-k)$-dimensional face $G$ of $\Pi_d$ corresponds with a $k$-chain in $\overline{\BB_{d+1}}:$
	\begin{equation}\label{equ:kchain} 
		\emptyset \subsetneq S_1 \subsetneq S_2 \subsetneq \cdots \subsetneq S_k \subsetneq [d+1] 
	\end{equation}
	such that the normal cone of $G$ is spanned by $\ee_{S_1}, \ee_{S_2}, \cdots, \ee_{S_k}.$

	For each $k$-chain in the form of \eqref{equ:kchain}, we associate with it the ordered set partition $\cT = T_1 | T_2 | \cdots | T_k | T_{k+1}$, where 
	\[ T_1 = [d+1] \setminus S_k, \quad T_2 = S_k \setminus S_{k-1}, \quad \dots, \quad T_{k} = S_2\setminus S_1, \quad T_{k+1} = S_1.\]
	One sees that this established a bijection between $k$-chains in $\overline{\BB_{d+1}}$ and ordered set partitions in $\overline{\OO_{d+1}},$ and hence induces a bijection between nonempty proper faces of $\Pi_d$ and ordered set partitions in $\overline{\OO_{d+1}}$. Furthermore, suppose $G$ is in bijection with ordered set partition $\cT$ through the $k$-chain \eqref{equ:kchain}. Then the chiseling direction of $\Pi_d$ at $G$ is
	\[ \sum_{i=1}^k \ee_{S_i} = \sum_{i=1}^k \sum_{j=k+2-i}^{k+1} \ee_{T_j} = \sum_{j=2}^{k+1} (j-1) \ee_{S_j}.\]
	As the normal fan is defined in $W_d$, where $\ee_{[d+1]}=\sum_{i\in[d+1]} \ee_i=0$, the above chiseling direction can be written as 
	\[ \sum_{j=2}^{k+1} (j-1) \ee_{S_j} + \ee_{[d+1]} = \sum_{j=1}^{k+1} j \ee_{S_j} = \ee_{\cT}.\]
\end{proof}

\begin{proof}[Proof of Theorem \ref{thm:bary2}]
	We apply Algorithm \ref{alg:bary} to $P_0 = \Pi_d$, making sure that the chiseling distances $\epsilon_i$'s satisfy
	\begin{equation}
	\epsilon_1 < \frac{1}{4}, \quad \text{and for $2 \le i \le d,$} \quad \epsilon_i < \frac{\epsilon_{i-1}}{2}.
	\label{equ:tiny2}
\end{equation}
Similar to the proof of Theorem \ref{thm:bary1}, we will show $P_d$ is a usual nested permutohedron. 

	First, $\Pi_d$ is a full-dimensional polytope in the affine space
	\[ 	\langle \ee_{[d+1]}, \xx \rangle = \langle \1, \xx \rangle \ = \ \sum_{j=1}^{d+1} j =: b_{[d+1]}.\]
	Next, for each $\cT \in \overline{\OO_{d+1}},$ let $G_\cT$ be its corresponding face of $\Pi_d$ assumed by Lemma \ref{lem:bij}. Then the chiseling direction of $G_\cT$ at $\Pi_d$ is $\ee_\cT.$ By Lemma \ref{lem:Pd}, the polytope $P_d$ is defined by the linear system: 
	\begin{equation} 
		\langle \ee_{[d+1]}, \xx \rangle = \langle \1, \xx \rangle \ = \ b_{[d+1]}, \quad \text{and} \quad 
		\langle \ee_\cT, \xx \rangle \ \le \ b_\cT - \epsilon_{d-k+1}, \quad \forall \cT \in \overline{\OO_{d+1}}, 
	\end{equation}
	where $b_{\cT} := \max_{\xx \in \Pi_d} \langle \ee_{\cT}, \xx \rangle.$
	Suppose $\Type(\cT)=(t_0, t_1,t_2,\dots, t_k, t_{k+1})$ (see Defintion \ref{defn:type} for the definition of structure type). Then we can compute that
	\[ b_{\cT} =  \left( \sum_{i=1}^{k+1} i \sum_{j=t_{i-1}+1}^{t_i} j\right).\]
	Note the above formula for $b_{\cT}$ not only works for $\cT \in \overline{\OO_{d+1}}$, but also works for $\cT = [d+1].$

	It follows from \eqref{equ:tiny2} that 
	\[
		0 < \epsilon_d < \epsilon_{d-1}-\epsilon_d < \epsilon_{d-2}-\epsilon_{d-1} < \cdots < \epsilon_1- \epsilon_2 < \epsilon_1 < \frac{1}{4}.
	\]
	Hence,
	\[ \bbeta := \Big( \epsilon_2-\epsilon_1, \epsilon_3 - \epsilon_2, \dots, \epsilon_d-\epsilon_{d-1}, -\epsilon_d\Big)\]
	is a strictly increasing sequence, where the absolute value of each entry is strictly smaller than $\frac{1}{4}.$ Therefore, letting $\balpha:=(1,2,\dots, d+1),$ one checks that $(M,N)=(1,1)$ is an appropriate choice for $(\balpha,\bbeta).$ Thus, it follows from Theorem \ref{thm:facetdes} that $P_d$ is the usual nested permutohedron $\Perm(\balpha, \bbeta;1,1).$
\end{proof}

\section{Questions}\label{sec:question}

We finish with questions that might be of interest for future research.

\begin{enumerate}[leftmargin=*]
\item Notice that in the case of the usual permutohedron $\Perm(\balpha),$ we always have $v_\pi^\balpha \in C(\pi)$ for each $\pi$, which is a property that makes notation natural. It is not always the case that for a usual nested permutohedron $\Perm(\balpha,\bbeta; M,N)$ we have $v_{\pi,\tau}^{(\balpha,\bbeta)(M,N)}\in C(\pi,\tau)$ for each $(\pi,\tau)$. Does there exist any usual or generalized nested permutohedron with this property?
\item Is there a way to realize the permuto-associahedron \cite[Lecture 9.3]{zie} as a deformation of the regular nested permutohedron?
\item The nested Braid fan was defined by grouping together points with the same relative order of coordinates and their first differences. One could go beyond and consider second differences, but this is \textbf{not} a subsequent barycentric subdivision. Is this ``doubly" nested Braid fan a projective fan?
\item The barycentric subdivision of a fan is obtained from stellar subdivisions in a particular order. If not done in the correct order the resulting fan is different. Which sequences of stellar subdivisions of $\Sigma(\Delta_d)$ give coarsenings of $\Br_d$?
\item As mentioned before, one of the motivations of this paper was to define and study a class of polytopes whose edges are parallel to directions in the form of $\ee_i+\ee_j-\ee_k-\ee_\ell$. The most direct way would be to first construct a fan from the hyperplane arrangement given by $x_i+x_j=x_k+x_\ell$ for all tuples $(i,j,k,\ell)$, including those with repeated elements, and then define a family of polytopes whose normal fans coarsen this new fan. One issue that arises is that this hyperplane arrangement is not simplicial. How many regions does it have? Do they have a combinatorial interpretation?
\end{enumerate}

\appendix
\section{Normal cones and projective fans}\label{apd:normal}

We will give a proof for Proposition \ref{prop:deform}, proceeded by 
definitions of normal cones and normal fans.
Recall that $W$ is the dual space of $V.$ Thus, any $\ww \in W$ can be considered as a linear functional on $V.$
\begin{definition}\label{defn:normal}
Suppose $P$ is a polytope in an affine space that is a translation of $V.$ Given a face $F$ of a polytope $P$, we define the \emph{normal cone} of $P$ at $F$:
\[
\ncone(F, P) 
:= \left\{ \ww \in W: \quad \langle \ww, \xx \rangle \geq \langle \ww, \yy \rangle, \quad \forall \xx \in F,\quad \forall \yy \in P \right\}.
\]
Therefore, $\ncone(F,P)$ is the collection of linear functionals $\ww$ in $W$ such that $\ww$ attains maximum value at $F$ over all points in $P.$

The \emph{normal fan} of $P$, denoted by $\Sigma(P),$ is the collection of all normal cones of $P$ as we range over all faces of $P$. 
\end{definition}

Since any linear functional $\ww \in W$ attains its maximum on some face of $P$, normal fans are always complete.

A fan $\Sigma'$ is a \emph{coarsening} of another fan $\Sigma$ if any cone in $\Sigma'$ is the union of a set of cones in $\Sigma.$ One can check that $\Sigma'$ is a coarsening of $\Sigma$ if and only if any maximal cone in $\Sigma'$ is the union of a set of maximal cones in $\Sigma.$

\begin{proof}[Proof of Proposition \ref{prop:deform}]
			Suppose $Q$ is a deformation of $P_0$. Then there exists $\bb \in \RR^m$ such that conditions \ref{item:pts} and \ref{item:nopass} of Definition \ref{defn:deform0} are satisfied. 
			Let $v, u$ be described as in condition \ref{item:nopass}. Then each of $\bfa_{i_1},\dots, \bfa_{i_k}$ attains maximum value at $u$ over all points in $Q,$ and thus is in $\ncone(u, Q).$ As $\bfa_{i_1},\dots, \bfa_{i_k}$ are spanning rays of $\ncone(v,P_0),$ we conclude that any maximal cone of $\Sigma(P_0)$ is a subset of some maximal cone of $\Sigma(Q)$. Since both $\Sigma(P_0)$ and $\Sigma(Q)$ are complete fans, we conclude that any maximal cone of $\Sigma(Q)$ is the union of a set of maximal cones in $\Sigma(P_0).$ Hence, $\Sigma(Q)$ is a coarsening of $\Sigma(P_0).$

Suppose $\Sigma(Q)$ is a coarsening of $\Sigma(P_0).$ 
	Let $\bb = (b_i)_{i=1}^m$ where 
\begin{equation}
b_i := \max_{\xx \in Q} \langle \bfa_i, \xx \rangle.
\label{equ:bi}
\end{equation}
We will show that $Q$ is the deformation of $P$ with the deforming vector $\bb$
by proving conditions \ref{item:pts} and \ref{item:nopass} of Definition \ref{defn:deform0} are satisfied.

Let $v$ be a vertex of $P_0$ that lies on facets $F_{i_1}, F_{i_2}, \dots, F_{i_k}$ of $P_0.$ Then $\{ \bfa_{i_1}, \dots, \bfa_{i_k}\}$ are the generating rays for $\ncone(v,P_0)$ and thus they belong to $\ncone(u, Q)$ for some vertex $u$ of $Q$. Therefore, 
\[
	\langle \bfa_{i_j}, u \rangle = \max_{\xx \in Q} \langle \bfa_{i_j}, \xx \rangle = b_{i_j}, \quad \forall 1 \le i \le k.
\]
Hence, condition \ref{item:nopass} of Definition \ref{defn:deform0} follows. 

Finally, let	\[ Q' := \{ \xx \in V \ : \ \A \xx \le \bb \}.\]
It is left to show that $Q' = Q.$ One sees that it is enough to show that any vertex $u$ of $Q$ is a vertex $Q',$ and $\ncone(u, Q) \subseteq \ncone(u,Q').$ Since $\Sigma(Q)$ is a coarsening of $\Sigma(P_0),$ we see any vertex $u$ of $Q$ arises in the way described in condition \ref{item:nopass}. Hence, $u$ is the intersection of hyperplanes determined by taking equalities of a subset of inequalities in $\A \xx \le \bb$ and satisfies the remaining inequalities. This implies $u$ is a vertex of $Q'.$ Let $R_u := \{ \bfa_{j_1}, \dots, \bfa_{j_t}\} \subseteq \{ \bfa_1, \dots, \bfa_m\}$ be the set of rays in $\ncone(u,Q).$ Then $\ncone(u,Q)$ is generated by rays in $R_u,$ and
\[ \langle \bfa_{j_s}, u \rangle = \max_{\xx \in Q} \langle \bfa_{j_s}, \xx \rangle = b_{j_s}, \quad \forall 1 \le s \le t.\]
Since $Q'$ is defined by $\A \xx \le \bb,$ this implies $\langle \bfa_{j_s}, \xx \rangle \le b_{j_s}$ for all $\xx \in Q$. Thus, $\langle \bfa_{j_s}, u \rangle = \max_{\xx \in Q'} \langle \bfa_{j_s}, \xx \rangle$ and so $\bfa_{j_s}$ belongs to $\ncone(u,Q')$ for each $1 \le s \le t.$ It follows that $\ncone(u,Q) \subseteq \ncone(u,Q').$
\end{proof}

\bibliography{biblio}

\begin{thebibliography}{10}

\bibitem{aguiar}
Marcelo Aguiar and Federico Ardila.
\newblock Hopf monoids and generalized permutahedra.
\newblock {\em arXiv:1709.07504}, 2017.

\bibitem{ardila}
Federico Ardila, Carolina Benedetti, and Jeffrey Doker.
\newblock Matroid polytopes and their volumes.
\newblock {\em Discrete Comput. Geom.}, 43(4):841--854, 2010.

\bibitem{barvinokconvex}
Alexander Barvinok.
\newblock {\em A course in convexity}, volume~54 of {\em Graduate Studies in
  Mathematics}.
\newblock American Mathematical Society, Providence, RI, 2002.

\bibitem{barvinok}
Alexander Barvinok.
\newblock {\em Integer points in polyhedra}.
\newblock Zurich Lectures in Advanced Mathematics. European Mathematical
  Society (EMS), Z\"urich, 2008.

\bibitem{bruns}
Winfried Bruns.
\newblock The quest for counterexamples in toric geometry.
\newblock In {\em Commutative algebra and algebraic geometry ({CAAG}-2010)},
  volume~17 of {\em Ramanujan Math. Soc. Lect. Notes Ser.}, pages 45--61.
  Ramanujan Math. Soc., Mysore, 2013.

\bibitem{cls}
David~A. Cox, John~B. Little, and Henry~K. Schenck.
\newblock {\em Toric varieties}, volume 124 of {\em Graduate Studies in
  Mathematics}.
\newblock American Mathematical Society, Providence, RI, 2011.

\bibitem{edmonds}
Jack Edmonds.
\newblock Submodular functions, matroids, and certain polyhedra.
\newblock In {\em Combinatorial {S}tructures and their {A}pplications ({P}roc.
  {C}algary {I}nternat. {C}onf., {C}algary, {A}lta., 1969)}, pages 69--87.
  Gordon and Breach, New York, 1970.

\bibitem{ewald}
G{\"u}nter Ewald.
\newblock {\em Combinatorial convexity and algebraic geometry}, volume 168 of
  {\em Graduate Texts in Mathematics}.
\newblock Springer-Verlag, New York, 1996.

\bibitem{fujishige}
Satoru Fujishige.
\newblock {\em Submodular functions and optimization}, volume~58 of {\em Annals
  of Discrete Mathematics}.
\newblock Elsevier B. V., Amsterdam, second edition, 2005.

\bibitem{grunbaum}
Branko Gr\"unbaum.
\newblock {\em Convex polytopes}, volume 221 of {\em Graduate Texts in
  Mathematics}.
\newblock Springer-Verlag, New York, second edition, 2003.
\newblock Prepared and with a preface by Volker Kaibel, Victor Klee and
  G\"unter M. Ziegler.

\bibitem{inequalities}
G.~H. Hardy, J.~E. Littlewood, and G.~P\'olya.
\newblock {\em Inequalities}.
\newblock Cambridge, at the University Press, 1952.
\newblock 2d ed.

\bibitem{kapranov}
Mikhail~M. Kapranov.
\newblock The permutoassociahedron, {M}ac {L}ane's coherence theorem and
  asymptotic zones for the {KZ} equation.
\newblock {\em J. Pure Appl. Algebra}, 85(2):119--142, 1993.

\bibitem{mcmullen}
Peter McMullen.
\newblock On simple polytopes.
\newblock {\em Invent. Math.}, 113(2):419--444, 1993.

\bibitem{ranktest}
Jason Morton, Lior Pachter, Anne Shiu, Bernd Sturmfels, and Oliver Wienand.
\newblock Convex rank tests and semigraphoids.
\newblock {\em SIAM J. Discrete Math.}, 23(3):1117--1134, 2009.

\bibitem{suho}
Suho Oh.
\newblock Generalized permutohedra, {$h$}-vectors of cotransversal matroids and
  pure {O}-sequences.
\newblock {\em Electron. J. Combin.}, 20(3):Paper 14, 14, 2013.

\bibitem{PosReiWil}
Alex Postnikov, Victor Reiner, and Lauren Williams.
\newblock Faces of generalized permutohedra.
\newblock {\em Doc. Math.}, 13:207--273, 2008.

\bibitem{post}
Alexander Postnikov.
\newblock Permutohedra, associahedra, and beyond.
\newblock {\em Int. Math. Res. Not. IMRN}, (6):1026--1106, 2009.

\bibitem{ReinerZieg}
Victor Reiner and G\"{u}nter~M. Ziegler.
\newblock Coxeter-associahedra.
\newblock {\em Mathematika}, 41(2):364--393, 1994.

\bibitem{schrijver}
Alexander Schrijver.
\newblock {\em Combinatorial optimization. {P}olyhedra and efficiency. {V}ol.
  {A}}, volume~24 of {\em Algorithms and Combinatorics}.
\newblock Springer-Verlag, Berlin, 2003.
\newblock Paths, flows, matchings, Chapters 1--38.

\bibitem{stanley-pitman}
Richard~P. Stanley and Jim Pitman.
\newblock A polytope related to empirical distributions, plane trees, parking
  functions, and the associahedron.
\newblock {\em Discrete Comput. Geom.}, 27(4):603--634, 2002.

\bibitem{zele}
Andrei Zelevinsky.
\newblock Nested complexes and their polyhedral realizations.
\newblock {\em Pure Appl. Math. Q.}, 2(3, Special Issue: In honor of Robert D.
  MacPherson. Part 1):655--671, 2006.

\bibitem{zie}
G\"unter~M. Ziegler.
\newblock {\em Lectures on polytopes}, volume 152 of {\em Graduate Texts in
  Mathematics}.
\newblock Springer-Verlag, New York, 1995.

\end{thebibliography}
\bibliographystyle{plain}

\end{document}